\documentclass[12pt, reqno]{amsart}

\usepackage{amssymb, amsmath, amsthm}
\usepackage[backref, colorlinks = true, linkcolor = blue, citecolor = blue, urlcolor=blue]{hyperref}
\usepackage[alphabetic,backrefs,lite]{amsrefs}
\usepackage{amscd}   
\usepackage{fullpage}
\usepackage[all]{xy} 
\usepackage{tikz-cd}
\usepackage{tikz}
\usepackage{comment}
\usepackage{colonequals}
\usepackage{MnSymbol}

\DeclareFontEncoding{OT2}{}{} 

\usepackage[]{xcolor}

\newtheorem{lemma}{Lemma}[subsection]
\newtheorem{theorem}[lemma]{Theorem}

\newtheorem{prop}[lemma]{Proposition}
\newtheorem{cor}[lemma]{Corollary}

\newtheorem{claim*}{Claim}
\newtheorem{thm}[lemma]{Theorem}

\theoremstyle{remark}
\newtheorem{remark}[lemma]{Remark}

\newtheorem{question}[lemma]{Question}
\newtheorem{example}[lemma]{Example}
\newtheorem{defn}[lemma]{Definition}

\newcommand{\A}{{\mathbb A}}

\newcommand{\PP}{{\mathbb P}}
\newcommand{\C}{{\mathbb C}}
\newcommand{\F}{{\mathbb F}}
\newcommand{\Q}{{\mathbb Q}}
\newcommand{\R}{{\mathbb R}}
\newcommand{\Z}{{\mathbb Z}}

\newcommand{\Xbar}{{\overline{X}}}
\newcommand{\Qbar}{{\overline{\Q}}}
\newcommand{\kbar}{{\overline{k}}}
\newcommand{\Fbar}{{\overline{F}}}
\newcommand{\Cbar}{{\overline{C}}}

\newcommand{\kk}{{\mathbf k}}

\newcommand{\calD}{{\mathcal D}}

\newcommand{\calL}{{\mathcal L}}

\newcommand{\calU}{{\mathcal U}}

\newcommand{\OO}{{\mathcal O}}

\newcommand{\frakp}{{\mathfrak p}}

\usepackage[mathscr]{euscript}

\DeclareMathOperator{\HH}{H}

\DeclareMathOperator{\rk}{rk}
\DeclareMathOperator{\Char}{char}

\DeclareMathOperator{\im}{im}

\DeclareMathOperator{\Hom}{Hom}

\DeclareMathOperator{\Gal}{Gal}

\DeclareMathOperator{\Res}{Res}

\DeclareMathOperator{\Br}{Br}

\DeclareMathOperator{\divv}{div}

\DeclareMathOperator{\Sym}{Sym}

\DeclareMathOperator{\Pic}{Pic}
\DeclareMathOperator{\Jac}{Jac}

\DeclareMathOperator{\Spec}{Spec}

\DeclareMathOperator{\rank}{rank}

\DeclareMathOperator{\hh}{h}
\DeclareMathOperator{\gon}{gon}
\DeclareMathOperator{\Stab}{Stab}
\DeclareMathOperator{\AV}{AV}

\newcommand{\Hilb}{\operatorname{Hilb}}
\newcommand{\mult}{\operatorname{mult}}

\newcommand{\degset}{\calD}
\newcommand{\dendegs}{\delta}
\newcommand{\potdendegs}{\wp}

\newcommand{\Ueno}{\operatorname{Ueno}}
\newcommand{\ind}{\operatorname{ind}}

\newcommand{\bd}{\mathbf{d}}

\numberwithin{equation}{subsection}
\numberwithin{table}{subsection}

\newcommand{\defi}[1]{\textsf{#1}} 
\newcounter{oldtocdepth}

\newcommand{\hidefromtoc}{%
  \setcounter{oldtocdepth}{\value{tocdepth}}%
  \addtocontents{toc}{\protect\setcounter{tocdepth}{-10}}%
}

\newcommand{\unhidefromtoc}{%
  \addtocontents{toc}{\protect\setcounter{tocdepth}{\value{oldtocdepth}}}%
}

\title{Isolated and parameterized points on curves}

\author{Bianca Viray}
\address{University of Washington, Department of Mathematics, Box 354350, Seattle, WA 98195, USA}
\email{bviray@uw.edu}
\urladdr{http://math.washington.edu/\~{}bviray}

\author{Isabel Vogt}
\address{Brown University, Department of Mathematics, Box 1917, 151 Thayer Street, Providence, RI 02912, USA}
\email{ivogt.math@gmail.com}
\urladdr{https://www.math.brown.edu/ivogt/}


\begin{document}

	\begin{abstract}
		We give a self-contained introduction to isolated points on curves and their counterpoint, parameterized points, that situates these concepts within the study of the arithmetic of curves.  In particular, we show how natural geometric constructions of infinitely many degree~\(d\) points on curves motivate the definitions of \(\PP^1\)- and AV-parameterized points and explain how a result of Faltings implies that there are only finitely many isolated points on any curve.  We use parameterized points to deduce properties of the density degree set and show that parameterized points of very low degree arise for a unique geometric reason. 
        The paper includes several examples that illustrate the possible behaviors of degree~\(d\) points.
	\end{abstract}

	\maketitle
	\hidefromtoc
\section{Introduction}
	Mordell's conjecture, proved by Faltings in 1983, asserts that any curve of genus at least~\(2\) has only finitely many points over any number field. This celebrated theorem is the prototypical example of the guiding philosophy of Diophantine geometry: ``geometry controls arithmetic''. In this example, the genus, a \emph{geometric} invariant which can be computed over~\(\mathbb{C}\), controls the arithmetic property of an infinitude of points over some number field.  

	While the infinitude of points over a number field is a very important arithmetic property, it is not the only one.  In some sense, the arithmetic of a curve \(C/\Q\) is encoded by all of its algebraic points \(C(\Qbar)\) equipped with an action of the absolute Galois group \(\Gal(\Qbar/\Q)\).  A Galois orbit in \(C(\Qbar)\) can also be viewed as a (scheme-theoretic) \defi{closed point} \(x\in C\).  Through this lens, the degree of the closed point is the size of the orbit (see Section~\ref{sec:schemey_def} for more details); in particular, the rational points are the same as the degree~\(1\) closed points.

	Considering closed points of arbitrary degree gives us more arithmetic information about \(C\).  So it is natural to ask: What is the analog of Mordell's Conjecture/Faltings's Theorem for degree~\(d\) points?  In other words, what geometric property of \(C\) controls the infinitude of degree~\(d\) points? The Mordell--Lang conjecture, proved by Faltings building on work of Vojta~\cites{Faltings-SubvarietiesOfAbelianVarieties, Faltings-GeneralLang, Vojta-Mordell}, implies that any infinite set of degree~\(d\) points occurs for a geometric reason: they arise in families parameterized by projective spaces or positive rank abelian varieties.  Bourdon, Ejder, Liu, Odumodu, and the first author~\cite{BELOV} observed that the Mordell--Lang Conjecture (applied repeatedly) implies the converse: any such parameterization gives an infinite collection of degree~\(d\) points. This led to the notion of isolated points, and in op.~cit., the authors study isolated points in the modular tower \(\{X_1(n)\}_{n\in \mathbb{Z}_{>0}}\). 
    
	The goal of this paper is to give a self-contained introduction to isolated points and their counterpoint, parameterized points, and situate these concepts in the study of the arithmetic of curves (not just modular curves).  We propose that the division of algebraic points into parameterized and isolated points gives a fruitful perspective for organizing and accessing the arithmetic of a curve, including characterizing sets of Zariski dense closed points.

    \subsection{Notation and conventions}\label{sec:Notation}

	Our primary interest is curves over number fields, so, for simplicity, we assume that our fields have characteristic \(0\).  We use \(F\) to denote an arbitrary characteristic \(0\) field, \(\Fbar\) to denote a fixed algebraic closure, and \(G_F\) to denote the absolute Galois group \(\Gal(\Fbar/F)\). We reserve \(k\) to denote a number field.

	In this paper, we use \defi{variety} to mean a separated scheme of finite type over a field \(F\).  A variety \(X/F\) is \defi{nice} if it is smooth, projective, and geometrically integral over~\(F\). For any extension \(F'/F\), we write \(X_{F'}\) for the base change of \(X\) to \(\Spec F'\) and write \(\Xbar := X_{\Fbar}\). A \defi{curve} is a variety of dimension \(1\) over a field. Throughout, we use \(C\) to denote a nice curve. 

    For any point \(x\in X\), we write \(\kk(x)\) for the residue field of \(x\) (i.e., \(R_{\frakp}/\frakp R_{\frakp}\), where \(x\) corresponds to the prime ideal \(\frakp\) in an affine open \(\Spec R \subset X\)); recall that \(\kk(x)/F\) is an extension with transcendence degree equal to the dimension of the Zariski closure of \(x\). 

    We use \(\subset\) and \(<\) to denote (not necessarily strict) containment of subsets and subgroups, respectively.
    \subsection{The main objects of study}
    \label{sec:Degreed}\label{sec:schemey_def}\label{sec:classical_def}
%
		Given a nice variety \(X\), a \defi{closed point} is a scheme theoretic point that is closed in the Zariski topology, or, equivalently, an integral \(0\)-dimensional subscheme of \(X\). From a classical viewpoint, we may think of a closed point as a Galois orbit of a geometric point, i.e., a closed point \(x\in X\) corresponds to an orbit \(O_x := \{\sigma(y) : \sigma\in G_F\}\) for some \(y = y_x\in X(\Fbar)\).  
		
		For any closed point \(x\in X\) its residue field \(\kk(x)\)  is isomorphic to the fixed field of the stabilizer \(\Stab_{\sigma(y)}\subset  G_F\) of some element \(\sigma(y)\in O_x\). Note that the fixed field of the stabilizer \(\Stab_{\sigma(y)}\subset  G_F\) is an \emph{embedded} field, i.e., a particular subfield of \(\Fbar\) so we have made a choice in how the abstract field \(\kk(x)\) is embedded in \(\Fbar\).  Different choices of this embedding correspond to different choices of elements \(\sigma(y)\in O_x\). The \defi{degree} of a closed point \(x\in X\) is the degree of this field extension \([\kk(x) : F]\); this is also equal to the cardinality of \(O_x\).
 
        \subsubsection{Degree sets, index, and density degrees}\label{sec:DegreeSets}
        There are several important invariants of a variety capturing the behavior of degree~\(d\) points.

        \begin{defn}
        Let \(X\) be a nice variety defined over a field \(F\).  
        \begin{enumerate}
            \item The \defi{degree set} \(\calD(X/F)\) is the set of all degrees of closed points on \(X\).
            \item The \defi{index} \(\ind(X/F)\) is the greatest common divisor of the integers in \(\calD(X/F)\).
            \item The \defi{density degree set} \(\dendegs(X/F)\) is the set of all positive integers \(d\) such that the set of degree~\(d\) points on \(X\) is Zariski dense.
            \item The \defi{minimum density degree}\footnote{This invariant was introduced by Smith and the second author~\cite{Smith--Vogt} and was previously referred to as the arithmetic degree of irrationality.} \(\min(\dendegs(X/F))\) is the smallest positive integer in \(\dendegs(X/F)\).
        \end{enumerate}
        \end{defn}
        \begin{remark}
        If \(k\) is a number field, then the strategy of  proof of \cite{GabberLiuLorenzini}*{Proposition 7.5} can be used to prove that \(\gcd(\dendegs(X/k)) = \ind(X/k)\), and so we do not need to introduce this as a separate invariant. We give the proof of this statement in Appendix~\ref{app:GCDdendegs}.
        \end{remark}

        We always have the containment \(\dendegs(X/F) \subset  \calD(X/F)\), and consequently, the index divides the minimum density degree.
        The sets \(\calD(X/F)\) and \(\dendegs(X/F)\) are sensitive to the ground field \(F\); for example, if \(F = \Fbar\), then they only contain the integer \(1\)!  More interestingly, if \(F\) is not algebraically closed, then there is no general containment between \(\calD(X/F')\) and \(\calD(X/F)\) under extensions \(F'/F\); for example, if \(X\) is a pointless conic over \(\R\), then \(\calD(X/\R) = \dendegs(X/\R) = \{2\}\) but \(\calD(X/\C) = \dendegs(X/\C) = \{1\}\).
        For this reason, we will also be concerned with the following geometric versions of these invariants that are insensitive to finite extensions.

\begin{defn}Let \(X\) be a nice variety defined over a field \(F\).
    \begin{enumerate}
        \item The \defi{potential density degree set} \(\potdendegs(X/F)\) is the union of \(\dendegs(X/F')\) as \(F'/F\) ranges over all possible finite extensions.
        \item The \defi{minimum potential density degree} is the smallest positive integer in  \(\potdendegs(X/F)\).
    \end{enumerate}
    \end{defn}

    \subsection{A first example: Quadratic points on hyperelliptic curves}\label{sec:ConstructionHyperellipticCurves}
        One of the main goals of the paper is to understand the geometry of how dense sets of degree~\(d\) points arise. Consider a curve \(C/\Q\) that is given by an equation of the form \(y^2 = f(x)\) for \(f(x)\) a separable polynomial of degree~\(2g + 2 \geq 6\); such a curve is hyperelliptic and has genus \(g\geq 2\).
	As discussed in the introduction, Faltings's Theorem (Theorem~\ref{thm:Faltings83}) tells us that \(C\) has only finitely many points over any number field \(k\).  \emph{However}, \(C\) evidently has infinitely many degree at most \(2\) points: for any choice of \(x\)-value in \(\Q\), solving for \(y\) requires at worst a quadratic extension. 
    Further, by Faltings's Theorem, there are only finitely many rational (i.e., degree~\(1\)) points  on \(C\), so the degree (exactly)~\(2\) points must be Zariski dense!  Thus \(2 \in \dendegs(C/k)\).
 
Since we can arrange for \(C\) to have any genus at least \(2\), the genus is \textbf{not} a fine enough geometric invariant to detect the Zariski density of degree~\(d\) points for \(d \geq 2\). Nevertheless, this infinite set of degree~\(2\) points has some evident geometric structure.  Intuitively, all of these points ``came from \(\mathbb{P}^1\)''.  We make this precise with the following definition.
		\begin{defn}\label{def:P1param}
			Let \(C\) be a nice curve over a number field \(k\).  
			
			A closed point \(x\in C\) is \defi{\(\PP^1\)-parameterized} if there exists a morphism \(\pi\colon C \to \PP^1\) with \(\deg(\pi) = \deg(x)\) and \(\pi(x)\in \PP^1(k)\).

			A closed point \(x\in C\) is \defi{\(\PP^1\)-isolated} if it is not \(\PP^1\)-parameterized.
		\end{defn}

Similarly to the \(d=2\) case above, Faltings's Theorem implies that if there is a degree~\(3\) morphism \(C\to \PP^1\), then \(3 \in \dendegs(C/k)\) and \(C\) has a degree~\(3\) \(\PP^1\)-parameterized point.  In the body of the paper, we will see how to extend this to larger degrees, as well as to study degree~\(d\) points that are not \(\PP^1\)-parameterized; these extensions will require tools coming from parameter spaces, the Mordell--Lang Conjecture and Hilbert's Irreducibility Theorem.

	\subsection{Organization of paper}
        We begin with background sections which give an overview of the objects, methods, and results used in the study of closed points.\footnote{This extensive background does defer the main body of the paper, which is a decision we made to make the paper as self-contained as possible. To mitigate this, on a first reading, the reader may wish to scan or skip these background sections and return to read in more detail after first gaining an appreciation of how these tools are used in the body of the paper.} In Section~\ref{sec:ToolsAllDim}, we introduce tools that can be used for studying closed points on varieties of arbitrary dimension, including the symmetric product and the Hilbert scheme and their relation to closed points, and Galois groups of closed points. In Section~\ref{sec:ToolsCurves}, we specialize to the case of curves where viewing closed points as divisors gives additional methods and perspectives. In particular, we review the relevant fundamentals of divisors, linear systems, and the Abel--Jacobi map and review the results of Faltings on rational points on curves and rational points on subvarieties of abelian varieties. In the last subsection (Section~\ref{sec:HilbIrred}), we review some results which show when finite maps to curves of genus \(0\) or \(1\) give rise to dense sets of degree~\(d\) points (the ingredients of the proof are discussed in further detail in Appendix~\ref{app:HilbertIrred}).

		Sections~\ref{sec:IntroToParam}--\ref{sec:SingleSource} form the body of the paper. In Section~\ref{sec:IntroToParam}, we describe constructions of infinitely many degree~\(d\) points, which we use as inspiration for the definitions of \(\PP^1\)-parameterized and AV-parameterized points.  We also define the complementary notions of \(\PP^1\)-isolated and AV-isolated points, which were first introduced in~\cite{BELOV}, and relate these notions to the so-called Ueno locus of a subvariety of an abelian variety. In Section~\ref{sec:Dichotomy}, we show how one of the results of Faltings together with the Riemann--Roch Theorem imply that there are only finitely many isolated points on any curve.  (Note that isolated and parameterized are properties that apply to \emph{any} closed point on a curve, of which there are always infinitely many.)  
        In Section~\ref{sec:DensityDegrees}, we use the definition of parameterized points to deduce properties of the density degrees and relate the minimum density degree to the gonality. In Section~\ref{sec:SingleSource}, we study parameterized points whose degree is very small (relative to the genus of the curve), and show that any such points stem from a single low degree morphism to another curve.

        The results in Sections~\ref{sec:IntroToParam}--\ref{sec:SingleSource} are often corollaries of prior work (which we indicate throughout) and foundational material. However, the perspective of parameterized points, isolated points, and density degree sets makes transparent properties that were not previously observed. In particular, to the best of our knowledge, Propositions~\ref{prop:covers_E}, \ref{prop:DensdegsClosedUnderScalarMult}, and~\ref{prop:dendegs_extension}, and Theorem~\ref{thm:single_source} are new.
   
        We close in Section~\ref{sec:Future} by highlighting some open questions and further directions.
\section*{Acknowledgements}
	Our perspectives on this topic were born in the joint works \cite{BELOV} and \cites{Smith--Vogt, Kadets--Vogt} respectively, and in discussions with audience members after reporting on that  work.
 The first author thanks her collaborators: Abbey Bourdon, \"{O}zlem Ejder, Yuan Liu, and Frances Odumodu. The second author thanks her collaborators, Geoffrey Smith and Borys Kadets, as well as Rachel Pries for recommending \cite{Debarre--Klassen} that piqued her interest in the area.

	The scope and content of the paper was clarified, honed, and polished through many conversations with others. We thank Tony Feng for pointing out that the recent work of Gao, Ge, and K\"uhne~\cite{GGK-UniformMordellLang} gives a uniform bound for the number of non-Ueno isolated points, and Ziyang Gao for helpful conversations about his joint work and for providing the reference in~\cite{ACGH} needed in the proof of Theorem~\ref{thm:Uniform}. We thank Nathan Chen for helpful discussions which led to Example~\ref{ex:DenDegsNotSemigroup} and for providing the reference~\cite{luroth} and Ariyan Javanpeykar for helpful discussions which led to Proposition~\ref{prop:Hilbert_E}. We thank James Rawson for the construction in Example~\ref{ex:rawson}, and thank Yufan Liu for pointing out an error in the previous version of Proposition~\ref{prop:SplittingUnderExtns} and Proposition~\ref{prop:dendegs_extension}. 
    We thank Maarten Derickx for helpful discussions about his work \cite{Derickx-SingleSource}, which is related to Theorem~\ref{thm:single_source}.
    We thank the participants of the \textit{Instructional workshop on rational points} in Groningen, the Netherlands,  \textit{Degree~d points on surfaces} workshop at AIM, and MATH581E: \textit{Algebraic points on curves} at the University of Washington for their questions and comments on the lectures based on this material. We also thank Abbey Bourdon, Borys Kadets, Filip Najman, Samir Siksek, Michael Stoll, and the anonymous referees for their comments on earlier drafts.

    The first author was supported in part by an AMS Birman Fellowship and NSF grant DMS-2101434, and thanks the UW ADVANCE's program \textit{Write Right Now} for providing excellent working conditions which facilitated the writing of this paper. The second author was supported in part by NSF grant DMS-2200655. 


\unhidefromtoc

\section{Tools for studying degree~\(d\) points on varieties}\label{sec:ToolsAllDim}
    \subsection{The Hilbert scheme and symmetric product}\label{sec:HilbSym}
        When studying degree~\(d\) points on a nice variety \(X/F\), one of the first reductions is to translate the problem to studying \emph{rational} points on an auxiliary variety \(Y_X\). The two standard choices of auxiliary varieties are the Hilbert scheme of points on \(X\) and the symmetric product of \(X\).
        \subsubsection{The Hilbert scheme of points}
        For any positive integer \(d\), there is a projective \(F\)-scheme, denoted \(\Hilb^d_X\), that is a fine moduli space for \defi{length}\footnote{The length of a \(0\)-dimensional scheme \(V/F\) is \(\dim_F H^0(V, \OO_V)\).} \(d\) subschemes on \(X\), i.e., 
        \begin{equation}\label{eq:hilb_functor_of_points}
        \Hilb^d_X(S) := \Hom_F(S, \Hilb^d_X) = \left\{\substack{\text{subschemes \(V \subset  X \times S\) proper flat over~\(S\)}\\ \text{ such that \(V\to S\) is finite of degree~\(d\)}}\right\}.
        \end{equation}
        This scheme \(\Hilb^d_X\) is commonly called the \defi{Hilbert scheme of \(d\) points} on \(X\) (although it would be more appropriate to call it the Hilbert scheme of length \(d\) subschemes of \(X\)).\footnote{More generally, for any polynomial \(P\) there exists a projective {Hilbert scheme} that parameterizes objects in \(X\) with Hilbert polynomial \(P\); see~ \cite{kollar-rational-curves}*{Chapter 1} for more details.} The functor of points \eqref{eq:hilb_functor_of_points} of \(\Hilb^d_X\) immediately implies a connection with degree~\(d\) points.
        
        \begin{lemma}\label{lem:HilbRationalPt}
            A degree~\(d\) point \(x \in X\) gives rise to a degree~\(1\) point on \(\Hilb^d_X\).
        \end{lemma}
        \begin{proof}
            By the definition of degree~\(d\) point given in Section~\ref{sec:schemey_def}, we have \([\kk(x):F] = d\).
            Hence \(x\) determines a map \(\Spec F \to \Hilb^d_X\) by \eqref{eq:hilb_functor_of_points}, which is a degree~\(1\) point on \(\Hilb^d_X\).
        \end{proof}
    
        On the other hand, applying \eqref{eq:hilb_functor_of_points} to \(S = \Hilb^d_X\), there is a universal length~\(d\) subscheme 
        \begin{center}
            \begin{tikzcd}
                 & X \times \Hilb^d_X  \arrow[dl, swap, "\pi_1"], \arrow[dr, "\pi_2"]& \calU^d_X \arrow[l, hook'] \arrow[d, "\pi_2|_{\calU^d_X}"] \\
                X && \Hilb_X^d
            \end{tikzcd}
        \end{center}
        over the Hilbert scheme of \(d\) points that corresponds to the identity in \(\Hilb_X^d(\Hilb^d_X)\).  The fibers of \(\pi_2|_{\calU^d_X}\) are precisely the possible length~\(d\) subschemes of \(X\).  So, as a partial converse to Lemma~\ref{lem:HilbRationalPt}, a degree~\(1\) point on \(\Hilb^d_X\) gives rise to a length~\(d\) subscheme of \(X\).

        \begin{remark}\label{rmk:HilbFunctor}
        The functor of points is very useful for studying geometric constructions of length~\(d\) subschemes, as we now explain. If \(\pi\colon X\to Y\) is a degree~\(d\) morphism, then for every \(y\in Y(k)\) the fiber \(\pi^*y = X_y\) gives a length~\(d\) subscheme of \(X\).  Moreover, if \(Y(F)\) is Zariski dense in \(Y\), then the set of fibers \(\{X_y : y \in Y(F)\}\) will be Zariski dense in \(X\). More precisely,~\eqref{eq:hilb_functor_of_points} gives an embedding \(Y\hookrightarrow \Hilb^d_X\), and \((\calU^d_X)_Y \to X\) dominates \(X\). 
    
        However, to determine whether \(Y(F)\) gives a Zariski dense set of \emph{degree~\(d\) points} (not only length~\(d\) subschemes) on \(X\), we must study the locus of \(Y\) where the 
        fibers of \(\pi\) are integral. In general, this is a difficult problem and will be a recurring theme throughout the paper.
        \end{remark}

        \subsubsection{The symmetric product}

        The \defi{\(d^{\text{th}}\) symmetric product} of \(X\) is the geometric quotient\footnote{This is not the stack quotient, since the action of \(S_d\) is not free.}
    	\[
    		\Sym^d_X := \underbrace{\left(X\times X\times \cdots \times X\right)}_d/S_d,
    	\]
    	where the symmetric group acts by permuting the \(d\) factors.  We write \(x_1 + \cdots + x_d\) for the image of the tuple \((x_1, \dots, x_d)\) in \(\Sym^d_X\). In general, \(\Sym^d_X\) need not be smooth, even if \(X\) is.  In fact, \(\Sym^d_X\) is singular whenever \(\dim X > 1\).

     In the case that \(X = \A^1 = \Spec F[t]\), the \(d^{th}\) symmetric product is the affine scheme given by the ring of invariants \(F[t_1, t_2, \dots, t_d]^{S_d}\) which is the ring generated by the \(d\) elementary symmetric functions. 
 Thus \(\Sym^d_{\A^1}\simeq \A^d\) and in particular is smooth.  Since smoothness can be checked \'etale locally, this implies that for any nice curve over a field \(F\), the symmetric product \(\Sym^d_C\) is a nice variety for all \(d\in \Z_{\geq1}\).  

        Given a degree~\(d\) point \(x\in X\), we may view \(x\) as a degree~\(1\) point on \(\Sym^d_X\). 
        Using the Galois action (since \(\Char(F)\) is assumed to be \(0\)), we may see this correspondence using the classical algebrogeometric description, where \(x\) is viewed as a Galois orbit \(O_x\subset  X(\Fbar)\) of cardinality \(d\).  By considering orderings of \(O_x\), we obtain \(d!\) points of \((X\times_F X\times_F\cdots \times_F X)(\Fbar)\) that are faithfully permuted by the symmetric group \(S_d\).  Thus, these orderings give a single point in the quotient \(\Sym^d_X\).  Since \(O_x\) was a Galois orbit, the image in \(\Sym^d_X\) is fixed by the absolute Galois group and thus must be a \(F\)-rational point of \(\Sym^d_X\).

        A more precise way to see this correspondence is via Lemma~\ref{lem:HilbRationalPt} and the so-called \defi{Hilbert--Chow} morphism
        \begin{align}
            \Hilb^d_X &\to \Sym^d_X \label{eq:HilbChow} \\
            x & \mapsto \sum_{z \in X} \mult_z(x) \cdot z \;. \nonumber
        \end{align}
        The open locus of \(\Sym^d_X\) consisting of \(d\) distinct geometric points is smooth (because it is a free quotient of a smooth variety) and the Hilbert--Chow morphism is a bijection on geometric points above this locus. Thus, a standard argument using Zariski's Main Theorem implies that the Hilbert--Chow morphism must be an isomorphism over this locus.\footnote{More precisely,  Zariski's Main Theorem~\cite{stacks-project}*{\href{https://stacks.math.columbia.edu/tag/05K0}{Lemma 05K0}} says that the Hilbert-Chow morphism must factor as an open immersion composed with a finite map. Since the morphism is also proper and a bijection above the locus \(U\) of distinct points, above \(U\)  the finite map must have degree \(1\) and the image of the composition is closed in \(U\). In other words, the Hilbert-Chow morphism is an isomorphism over \(U\).}

        If \(X = C\) is a nice curve, then all of \(\Sym^d_C\) is smooth (as explained above) and the Hilbert--Chow morphism is a bijection on geometric points over the entire space.\footnote{The inverse map on geometric points sends an unordered \(d\)-tuple of geometric points to the appropriate powers of the maximal ideals corresponding to the distinct points.} Thus, the same argument proves that Hilbert--Chow morphism gives an isomorphism \(\Hilb^d_C\xrightarrow{\sim}\Sym^d_C\) (and thus \(\Hilb^d_C\) is smooth). In what follows we will simply refer to this this variety as \(\Sym^d_C\) to match conventions in the literature.
 
        Even though a degree~\(d\) point on \(X\) gives rise to an \(F\)-point on \(\Sym^d_X\), the converse is not necessarily true (e.g., \(dx \in \Sym^d_X(F)\) if \(x \in X(F)\)).  We can characterize the degree~\(d\) points in \(\Sym^d_X(F)\) using morphisms of symmetric products.
        Let \(\bd = (d_1, \dots, d_r)\) be any partition of the positive integer \(d\) (we write \(\bd \vdash d\)).  Let \(\Sym^{\bd}_X\) denote the product \(\prod_{d_i \in \bd}\Sym^{d_i}_X\). There is a \defi{summation map}
        \[
        \sigma_{\bd} \colon \Sym^{\bd}_X\to \Sym^d_X
        \]
        sending the tuple \((y_1, \dots, y_r)\) with \(\deg(y_i) = d_i\) to the sum \(y_1 + \cdots + y_r\), which has degree~\(d\) by assumption.

\begin{lemma}\label{lem:image_summation_reducible}
     The degree~\(d\) points on a nice variety \(X/F\) correspond precisely to the complement of \(\bigcup_{\bd}\sigma_{\bd}\left(\Sym^{\bd}_X(F)\right)\) in \(\Sym^d_X(F)\) as \(\bd\) ranges over all nontrivial partitions of \(d\).
 \end{lemma}
 \begin{proof}
     Let \(z\in \Sym^d_X(F)\). 
     Taking a geometric point in the fiber of  \(X\times \dots \times X \to \Sym^d_X\) over \(z\) and composing with each projection, we obtain \(d\) (not necessarily distinct) geometric points \(x_1, \dots, x_d\in X(\Fbar).\)  Since the absolute Galois group fixes the \(S_d\) orbit of these points, the points \(x_1, \dots, x_d\) can be partitioned into complete Galois orbits \(O_1, \dots, O_r\), where \(\bd := (\#O_i)_{i=1}^r\) gives a partition of \(d\). (Note that pairwise these orbits are either distinct or coincide.)
     Each orbit gives a point \(y_i\in \Sym^{d_i}(F)\) such that \(\sigma_{\bd}((y_i)) = z\). If \(z\) is not in the image of \(\sigma_{\bd'}\) for any nontrivial partition \(\bd'\), then the aforementioned partition must be trivial, i.e., \(\bd = (d)\). In other words, the \(x_i\) are pairwise distinct and the absolute Galois group acts transitively on them. This is equivalent to \(z\) corresponding to a degree~\(d\) point.
 \end{proof}
     \begin{remark}
         If the base field has characteristic \(p\), then this lemma can fail; see~\cite{Ma-CubicHypersurfaces}*{Example 4.4 and Proposition 4.5}.
     \end{remark}

\subsection{Residue fields and Galois groups}\label{subsec:GaloisTheory}
    Given a degree~\(d\) point \(x \in X\), we have its residue field \(\kk(x)\), which is a degree~\(d\) extension of \(F\), and its \defi{Galois group} (or rather the Galois group of the Galois closure) which is a transitive subgroup \(G\) of \(S_d\) well-defined up to conjugacy. From the classical algebrogeometric point of view, this is the image of the permutation representation of the absolute Galois group \(G_F\) on the orbit \(O_y\) of a geometric point \(y \in X(\Fbar)\). We will refer to \(G\) as the Galois group of the point \(x\). 
 
    It is also natural to consider the Galois group of a reduced \(0\)-dimensional subscheme of \(X\) that is not necessarily assumed to be integral.  This can again be described as the image of a permutation representation of \(G_F\) or as the Galois group of an \'etale algebra (see \cite{Poonen-RatlPoints}*{Section 1.3} for Galois theory in this context).  The difference is that in this generality, the group \(G\) is not necessarily a transitive subgroup of \(S_d\).  

        Given a closed point \(x\in X\) and an extension \(F'/F\), the base change \(x_{F'}\) can fail to remain a point by becoming reducible. Indeed, for any choice of embedding \(\kk(x)\hookrightarrow \overline{F'}\), the base change \(x_{F'}\) contains a point of degree equal to the degree of the compositum \(F'_x \colonequals F' \cdot \kk(x)\) over \(F'\), and this degree may be smaller than \([\kk(x): F]\). In other words, the base change \(x_{F'}\) is integral if and only if \(F'/F\) is linearly disjoint from \(\kk(x)/F\).

        If \(x_{F'}\) is not integral, then Galois theory shows that the decomposition type is constrained.
        \begin{prop}\label{prop:SplittingUnderExtns}
            Let \(F'/F\) be a finite extension and let \(F''\) denote its Galois closure. Let \(x\in X\) be a degree~\(d\) point.
            \begin{enumerate}
                \item All components of \(x_{F''}\) have the same degree~\(n\) and \(d\equiv 0 \bmod n\).\label{part:Galois}
                \item If \(x_{F'} \in \sigma_{\bd}\left(\Sym^{\bd}_X(F')\right)\) for some partition \(\bd\vdash d\), then \(\gcd(\bd)\equiv 0 \bmod n\).\label{part:GaloisGCD}
                \item If \(x_{F'} \in \sigma_{\bd}\left(\Sym^{\bd}_X(F')\right)\) for some partition \(\bd \vdash d\) with \(\gcd(\bd) = 1\), then \(\kk(x)\subset  F''\).\label{part:gcd1}
            \end{enumerate}
        \end{prop}
        \begin{proof}
            Since \(F''/F\) is Galois, the Galois group \(\Gal(F''/F)\) acts on the components of \(x_{F''}\) and thus they must all have the same degree, which we denote \(n\). Since the degrees of the components of \(x_{F''}\) sum to \(\deg(x)\), \(d\) is divisible by \(n\), proving~\eqref{part:Galois}.  

            Assume that \(x_{F'} \in \sigma_{\bd}\left(\Sym^{\bd}_X(F')\right)\) and let \(y_i\in X_{F'}\) be such that \(\deg(y_i ) = d_i\) and \(\sum_i y_i = x_{F'}\). By applying part~\eqref{part:Galois} with \(x = y_i\) and the extension \(F''/F'\), we see that \(d_i\equiv 0 \bmod n\), and so \(\gcd(d_i) \equiv 0 \bmod n.\) This proves~\eqref{part:GaloisGCD}. Now assume that \(\gcd(d_i) = 1\).  Then \(n\) must equal \(1\) and so \(x_{F''}\) is a sum of \(F''\)-rational points. In other words, \(\kk(x) \subset  F''\), giving~\eqref{part:gcd1}.
        \end{proof}

\section{Additional tools for studying degree~\(d\) points in the case of curves}\label{sec:DegreedCurves}\label{sec:ToolsCurves}
 
    If \(X = C\) is a nice curve, then a degree~\(d\) point is also a degree~\(d\) effective divisor, and considering linear equivalences of divisors endows \(\Sym^d_C=\Hilb^d_C\) with additional structure.

    We write \(\Pic C\) for the Picard group of \(C\), and \(\Pic_C\) for the Picard scheme. Recall that \(\Pic_C(F) = (\Pic\Cbar)^{G_F}\) which can, in general, be larger than \(\Pic C\). For an integer \(d\), we write \(\Pic^d_C\) for the connected component consisting of equivalence classes of degree~\(d\) divisors.  Recall that \(\Pic^0_C\) is a \(g\)-dimensional abelian variety where \(g := \hh^1(C, \OO_C)\) is the \defi{genus} of the curve. The map sending effective divisors to their linear equivalence class induces the degree~\(d\) \defi{Abel--Jacobi morphism} \(\Sym^d_C \to \Pic^d_C\). This will be a fundamental tool in our study. 
    
    \subsection{Fibers of the Abel-Jacobi map and maps to projective space}

        To any divisor \(D\) on \(C\), we may associate its \defi{Riemann--Roch space}, i.e., the \(F\)-vector space
        \[\HH^0(C, D) \colonequals \{f \in \kk(C)^\times : \divv(f) + D \text{ is effective}\} \cup \{0\}.\]
        A divisor of the form \(\divv(f) + D\) for \(f \in \kk(C)^\times\) is \defi{linearly equivalent} to \(D\), and we write \([D]\) for the linear equivalence class of \(D\).
        Since \(\HH^0(C, D)\) depends only on the linear equivalence class of \(D\), we also write \(\HH^0(C, [D]) = \HH^0(C, D)\).  Let \(\hh^0([D]) = \hh^0(D) \colonequals \dim_F \HH^0(C, D)\). As the notation suggests, this vector space can be identified with cohomology, and as such the function \([D]\mapsto\hh^0([D])\) is upper semicontinuous on \(\Pic^d_C\)~\cite{Hartshorne}*{Chapter III, Theorem 12.8 and Chapter II, Section 7}.
        The \defi{complete linear system of \(D\)}, denoted \(|D|\), is the set of all effective divisors linearly equivalent to \(D\), and is isomorphic to a projective space of dimension \({\hh^0(D) - 1}\). By definition, the fiber of the Abel--Jacobi morphism over~\([D] \in \Pic(C)\) is the complete linear system \(|D|\). A nonempty linear system \(|D|\) is called \defi{basepoint free} if the intersection of the supports of all divisors in \(|D|\) is empty.
        See~\cite{Hartshorne}*{Section~IV.1}, or~\cite{Silverman-AEC}*{Section II.5} for more details.
        \begin{remark}
            The vector space \(\HH^0(C, D)\) and its dimension \(\hh^0(D)\) are sometimes denoted by \(\calL(D)\) and \(\ell(D)\) respectively. However, \(\calL(D)\) is sometimes instead used to denote the line bundle associated to \(D\). To avoid this confusion, we use \(\HH^0(C, D)\) and \(\hh^0(D)\) throughout.
        \end{remark}

        A basepoint free linear system \(|D|\) gives rise to a nondegenerate\footnote{A morphism \(C \to \PP^r\) is called nondegenerate if the image of \(C\) spans \(\PP^r\).} morphism \(C \to \PP^r\) such that the preimage of a hyperplane in \(\PP^r\) is an effective divisor linearly equivalent to \(D\)
        (see~\cite{Hartshorne}*{Section~II.7}). 
        The following lemma shows that a morphism to a projective space of dimension at least \(2\) can be composed with a projection to yield a morphism of the same degree to \(\PP^1\).
        
        \begin{lemma}\label{lem:projection}
            For \(r \geq 2\), let \(\varphi \colon C \to \PP^r\) be a nondegenerate morphism and let \(H \subset \PP^r\) be a hyperplane defined over \(F\). There exists a morphism \(\pi \colon C \to \PP^1\) and a point \(p \in \PP^1(F)\) such that \(\pi^{-1}(p) = \varphi^{-1}(H)\).
        \end{lemma}
        \begin{proof}
            As \(\varphi\) is nondegenerate, the intersection \(\varphi(C) \cap H\) is \(0\)-dimensional.
            Therefore, inside the dual projective space \(H^\vee \simeq \PP^{r-1}\) of \((r-2)\)-planes contained in \(H\), the locus of \((r-2)\)-planes \(\Lambda\) such that \(\varphi^{-1}(\Lambda) \neq \emptyset\) has codimension \(1\).
            Hence, since \(F\) is infinite, there exists an \((r-2)\)-plane \(\Lambda \subsetneq H \subset \PP^r\) such that \(\varphi^{-1}(\Lambda) = \emptyset\).  Projection \(\pi_\Lambda\) from \(\Lambda\) defines a morphism \(\pi_\Lambda \colon \PP^r \smallsetminus \Lambda \to \PP^1\); intrinsically, the codomain \(\PP^1\) is the pencil of \((r-1)\)-planes in \(\PP^r\) containing \(\Lambda\).
            Composing \(\varphi\) with projection \(\pi_\Lambda\) from \(\Lambda\) gives a morphism \(\pi \colonequals\pi_\Lambda \circ \varphi \colon C \to \PP^1\).  Since \(H\) is an \((r-1)\)-plane containing \(\Lambda\), it corresponds to a point \(p \in \PP^1(F)\), and \(\pi^{-1}(p) = \varphi^{-1}(H)\).
        \end{proof}

        Said another way, Lemma~\ref{lem:projection} guarantees that if \(|D|\) is basepoint free, 
        then there are two disjoint effective divisors \(D', D''\in |D|\). We freely use this implication throughout.  Moreover, summation of divisors together with Lemma~\ref{lem:projection} implies that the set of degrees of morphisms from \(C\) onto \(\PP^1_F\) has the structure of an additive semigroup; this is called the \defi{L\"uroth semigroup} of \(C/F\)  \cite{luroth}*{Definition-Proposition 1}.
    	The \defi{gonality} of \(C/F\) is the minimal element of the L\"uroth semigroup, i.e.,
    	\[
    		\gon(C)  := \min\{\deg \phi : \phi\colon C \twoheadrightarrow \PP^1_F\}.	
    	\]
    	Note that for any extension \(F'/F\), we have \(\gon(C) \geq \gon(C_{F'})\).  See~\cite{Poonen-Gonality}*{Theorem 2.5} for further restrictions on the gonality under base extensions.

        \subsection{The image of the Abel--Jacobi map and its rational points}\label{sec:AJ}
        Let \(W_C^d\subset \Pic^d_C\) denote the image of the degree~\(d\) Abel--Jacobi map, i.e., \(W_C^d = \im (\Sym^d_C \to \Pic^d_C)\). If \(g \geq 1\), then by the Riemann--Roch Theorem~\cite{Hartshorne}*{Theorem 1.3}, the Abel--Jacobi map yields the isomorphisms \(W_C^1 \simeq C\) and for \(d\geq g\), \(W_C^d \simeq \Pic^d_C\).  In general, \(W_C^d\) is an irreducible subvariety of dimension \(\min(d,g)\).

        The \(F\)-rational points of \(\Sym^d_C\) map to \(F\)-rational points of \(W^d_C\), but this map is not necessarily surjective (even though \(\Sym^d_C\to W^d_C\) is a surjective map of schemes). In general, \(\im(\Sym^d_C(F) \to W^d_C(F)) = W^d_C(F) \cap \Pic(C) \subset W^d_C(F)\cap \Pic^d_C(F) = W^d_C(F)\). If we consider an arbitrary point \(w \in W_C^d(F)\), then the fiber of the Abel--Jacobi map over \(w\) is a priori only a Severi--Brauer variety of dimension \(\hh^0(w)-1\). However, since there are no nontrivial \(0\)-dimensional Severi--Brauer varieties, we have the following:
        \begin{lemma}\label{lem:BrSev}
            If \([D]\in \Pic_C^d(F)\) and \(\hh^0([D]) = 1\), then \([D]\in \im (\Sym_C^d(F) \to \Pic_C^d(F))\).
        \end{lemma}
        \begin{proof}
            Since \(\hh^0([D]) = 1\), geometrically we have \(|D|=\{E\}\) for a unique effective divisor \(E\) defined over \(\bar{F}\). Since, further, \([D]\) gives a rational point of \(\Pic_C^d\), for every \(\sigma\in G_F\), we have that \(\sigma(E)\) is linearly equivalent to \(E\).  Since \(|E| = |D|=\{E\}\), we conclude that \(\sigma(E) = E\).  Thus \(E\) gives a point in \(\Sym_C^d(\Fbar)^{G_F} = \Sym_C^d(F)\).
        \end{proof}

        If \(\Pic^d_C\) has no rational points, then neither does \(W_C^d\) nor \(\Sym^d_C\).  However, if \(\Pic^d_C\) does have a rational point, then \(\Pic^d_C\simeq \Pic^0_C\) and so we may view \(W_C^d\) as a subvariety of an abelian variety.  If \(F = k\) is a number field, then we may apply the following landmark theorem of Faltings to obtain a structural description of the rational points of \(W_C^d\); this description will play a key role throughout the paper.
        \begin{thm}[Faltings's Subvarieties of AV Theorem~\cites{Faltings-SubvarietiesOfAbelianVarieties, Faltings-GeneralLang}]\label{thm:Faltings91}\hfill
        
            Let \(W \subset P\) be a closed subvariety of an abelian variety over a number field \(k\).
            Then 
            \[W(k) = \bigcup_{i=1}^N (w_i + A_i(k)),\]
            where \(w_i\in W(k)\) and  $A_i$ are abelian subvarieties \(A_i \subset P\)  such that \(w_i + A_i \subset W\).
        \end{thm}
        In the case of \(W_C^1\simeq C\), this recovers Faltings's earlier result proving the Mordell conjecture.
        \begin{thm}[Faltings's Curve Theorem~\cite{Faltings-Mordell}*{Satz 7}]\label{thm:Faltings83}
            Let \(C\) be a nice curve of genus at least~\(2\) defined over a number field \(k\).  Then \(C(k)\) is finite.
        \end{thm}
        
        \begin{remark}
            In the literature, Theorems~\ref{thm:Faltings91} and~\ref{thm:Faltings83} are often referred to as the Mordell--Lang conjecture (or Lang Conjecture) and the Mordell Conjecture, respectively.  Since we will refer to these results extensively in the paper and using the term `conjecture' to refer to results that are proved can be confusing to newcomers to the field, we have elected to refer to them as ``Faltings's Subvarieties of AV Theorem'' and ``Faltings's Curve Theorem'', even though this is not standard convention.  
        \end{remark}
        Over number fields, a large subset of \(W_C^d(k)\) must be contained in the image of \(\Sym^d_C(k)\).
        \begin{lemma}\label{lem:FiniteIndexEffective}
        Let \(C\) be a nice curve over a number field \(k\), let \(E\) be an effective degree~\(d\) divisor on \(C\), and let \(A\subset \Pic^0_C\) be a positive rank abelian variety such that \({[E]} + A\subset W_C^d\).  Then there is a finite index subgroup \(B < A(k)\) such that \({[E]} + B \subset \im (\Sym^d_C(k) \to W_C^d(k))\).
    \end{lemma}
    \begin{proof}
        Since \({[E]} + A(k)\subset W_C^d\), every divisor \emph{class} in \({[E]} + A(k)\) is represented by a \(\kbar\)-effective divisor.  The divisor classes of \({[E]} + A(k)\) that are contained in the image of \(\Sym^d_C(k) \to W_C^d(k)\) are exactly the divisor classes that have \(k\)-rational representatives.  In other words,
        \begin{align*}
            ({[E]} + A(k))\cap \im (\Sym^d_C(k) \to W_C^d(k)) & = ({[E]} + A(k))\cap \im (\Pic(C) \to \Pic_C(k))\\
            & = {[E]} + \left(A(k)\cap \im (\Pic(C) \to \Pic_C(k)\right).
        \end{align*}
        The exact sequence of low degree terms from the Hochschild-Serre spectral sequence (see~\cite{Poonen-RatlPoints}*{Section 6.7, Corollary 6.7.8}) shows that 
        \[
            A(k)\cap \im (\Pic(C) \to \Pic_C(k)) = A(k)\cap \ker (\Pic_C(k)\to \Br k) = \ker (A(k)\to \Br k).
        \]
        Since \(A(k)\) is finitely generated and \(\Br k\) is torsion, we obtain the desired result.
    \end{proof}
        
    
    \subsubsection{Rational points of \(W^d_C\) that correspond to degree~\(d\) points}
        Descending the summation map from Lemma~\ref{lem:image_summation_reducible} to a map on \(\Pic_C^d\) gives a sufficient condition for an effective divisor class to be represented by a degree~\(d\) point. Writing \(W_C^{\bd} \subset \Pic^{\bd}_C\) for the product \(W_C^{d_1} \times \cdots \times W_C^{d_r} \subset \Pic^{d_1}_C \times \cdots \Pic^{d_r}_C\), the summation map \(\sigma_{\bd}\) fits into a commutative diagram map with the group law\footnote{Recall that the group law on \(\Pic_C\) corresponds to tensor products of line bundles, which agrees with the sum of divisors/divisor classes.}
 \(+\) on \(\Pic_C\):

        \begin{center}
            \begin{tikzcd}
                \Sym^{\bd}_C \arrow[d, "\sigma_{\bd}"] \arrow[r, "\prod \text{Abel--Jacobi}"] &[5em] W_C^{\bd} \arrow[r, hook] \arrow[d] & \Pic^{\bd}_C \arrow[d, "+"] \\
                \Sym^d_C \arrow[r, "\text{Abel--Jacobi}"]& W_C^d \arrow[r, hook] & \Pic^d_C
            \end{tikzcd}
        \end{center}
        \begin{cor}[Corollary of Lemma~\ref{lem:image_summation_reducible}]\label{cor:image_summation_reducible}
             Let \([D]\in \im (\Sym^d_C(F) \to W_C^d(F))\). If \([D]\notin\im\left(+ \colon W_C^{\bd}\to W_C^d\right)\) for any nontrivial \(\bd\vdash d\), then \([D]\) is represented by a degree~\(d\) point.
         \end{cor}

\subsection{Integrality of finite fibers and Hilbert's Irreducibility Theorem}\label{sec:HilbIrred}

    In our last background section, we return to the question raised in Remark~\ref{rmk:HilbFunctor} specialized to curves over number fields: given a degree~\(d\) morphism \(\pi\colon C\to Z\), where \(Z\) is a curve with a Zariski dense set of \(k\)-points, when is the set
    \(
    \{z\in Z(k) : C_z \text{ integral}\}
    \)
    dense in \(Z\)?
    If \(Z = \PP^1\), then Hilbert's Irreducibility Theorem can be used to prove that the answer is always. 
    \begin{prop}[Specialization of Proposition~\ref{prop:hilbert}]\label{prop:hilbert-curve}
        Let \(C\) be a nice curve over a number field \(k\), and let \(\pi \colon C \to \PP_k^1\) be a degree~\(d\) morphism. Then there exists a Zariski dense subset of points \(z \in \PP^1(k)\) such that \(C_z\) is a degree~\(d\) point on \(C\).  
        In particular, \(d \in \dendegs(C/k)\).
    \end{prop}
    \begin{remark}
        Giving the full proof of Proposition~\ref{prop:hilbert-curve} would take us far afield from the topics discussed in this paper. At the same time, many (but not all) of the ingredients of the proof are similar to ideas we have already discussed thus far. For this reason, we provide  Appendix~\ref{app:HilbertIrred} where we sketch of the main ingredients of the proof  of Proposition~\ref{prop:hilbert-curve} and give the proof of Proposition~\ref{prop:Hilbert_E} below.
    \end{remark}

    If \(Z\) is not rational, then \emph{all} fibers \(C_z\) can fail to be integral.
For example, there exist elliptic curves \(E\) of positive rank and nontrivial isogenies \(\pi \colon E' \to E\), for which \(E(k) = \pi(E'(k))\); for example: the \(3\)-isogeny from \href{https://www.lmfdb.org/EllipticCurve/Q/175/b/3}{175.b3} to \href{https://www.lmfdb.org/EllipticCurve/Q/175/b/1}{175.b1} has this property. 
%
In such examples, the locus of points in \(E(k)\) for which the fiber is integral is empty.  In the same way, if \(C \to E\) is any cover that factors \(C \to E' \to E\) through a nontrivial isogeny \(E'\to E\) as above, then the conclusion of Proposition~\ref{prop:hilbert} fails for this cover.\footnote{While the conclusion fails over \(k\), Tucker proves there exists a finite extension \(k'/k\) over which the conclusion holds~\cite{Tucker}*{Theorem 2.5}.}  On the other hand, ideas from~\cite{cdjlz} imply that factoring through a nontrivial connected \'etale cover (such as an isogeny) is the \emph{only} way that Proposition~\ref{prop:hilbert} can fail for a cover of a positive rank elliptic curve.

\begin{prop}[c.f.~\cite{cdjlz}*{Theorem 1.4}]\label{prop:Hilbert_E}
Let \(C\) be a nice curve defined over a number field \(k\) and let \(E\) be an elliptic curve of positive rank.  If \(\pi \colon C \to E\) is a degree~\(d\) cover that does not factor through any nontrivial connected \'etale subcover \(C' \to E\), then there exists a Zariski dense locus of points \(t \in E(k)\) for which \(\pi^{-1}(t)\) is a degree~\(d\) point on \(C\).
\end{prop}

\section{Parameterizing degree~\(d\) points on curves}\label{sec:IntroToParam}

	\subsection{\(\PP^1\)-parameterized points}
	As we saw in Section~\ref{sec:ConstructionHyperellipticCurves} (see Definition~\ref{def:P1param}), we say that a closed point \(x\in C\) is \(\PP^1\)-parameterized if there exists a morphism \(\pi\colon C \to \PP^1\) with \(\deg(\pi) = \deg(x)\) and \(\pi(x) \in \PP^1(k)\).		Intuitively, \(x\) is \(\PP^1\)-parameterized if there is a rational parameterization of a subset of degree~\(d\) points (more precisely, a subvariety of \(\Sym^d_C\)) that  contains \(x\).  Unraveling this intuition gives several equivalent definitions:
    \begin{lemma}\label{lem:P1paramDefineorem}
        Let \(C\) be a nice curve over a number field \(k\) and let \(x\in C\) be a closed point of degree~\(d\).  The following are equivalent:
        \begin{enumerate}
            \item The point \(x\) is \(\PP^1\)-parameterized.\label{cond:P1param}
            \item There exists a nonconstant map \(\PP^1 \to \Sym^d C\) whose image contains \(x\in (\Sym^dC)(k)\).\label{cond:P1inSymd}
            \item The dimension \(\hh^0([x])\) is at least \(2\), where \([x]\) denotes the divisor class of \(x\).\label{cond:linearsystem}
            \item The linear system \(|x|\) is basepoint free.\label{cond:bpf}
        \end{enumerate}
    \end{lemma}

    Condition~\eqref{cond:linearsystem} is intrinsically defined and gives a method of testing whether a given point is \(\PP^1\)-parameterized or \(\PP^1\)-isolated. (In algebraic geometry terminology, condition~\eqref{cond:linearsystem} can be equivalently phrased that the divisor \(x\) \defi{moves}.) Upper semicontinuity of \(\hh^0\) means that if there is a single \(\PP^1\)-isolated point of degree \(d\) then there is an open dense subset of \(W^d\) over which \(\hh^0(D) = 1\). On the other hand, the Riemann--Roch Theorem and condition~\eqref{cond:linearsystem} imply that all points of sufficiently large degree are \(\PP^1\)-parameterized. 
    \begin{cor}\label{cor:HighDegreeP1Param}
       Let \(C\) be a nice curve over a number field \(k\) and let \(x\in C\) be a closed point of degree~\(d\). If \(d\geq g+1\), then \(x\) is \(\PP^1\)-parameterized.
    \end{cor}
    \begin{proof}
       By the Riemann--Roch Theorem, \(\hh^0([x]) \geq d + 1 - g\), which by assumption is at least~\(2\). So \(x\) is \(\PP^1\)-parameterized by Lemma~\ref{lem:P1paramDefineorem}.
    \end{proof}
    
    \begin{cor}\label{cor:BijP1isolated}
       The Abel--Jacobi map gives a bijection between the set of degree~\(d\) \(\PP^1\)-isolated points and the set of their divisor classes.
    \end{cor}
    \begin{proof}
       By Lemma~\ref{lem:P1paramDefineorem}, any \(\PP^1\)-isolated point \(x\) has \(\hh^0([x]) = 1\), so there is a unique effective representative of \([x]\); equivalently, the fiber above \([x]\) in the Abel--Jacobi map \(\Sym^d_C\to \Pic^d_C\) is exactly \(x\).
    \end{proof}   
    \begin{proof}[Proof of Lemma~\ref{lem:P1paramDefineorem}]
        Assume~\eqref{cond:P1param}.  Then there exists a degree~\(d\) morphism \(\pi\colon C\to \PP^1\) that sends \(x\) to a rational point.  Thus the morphism \(\PP^1\to \Sym^d C, t\in \PP^1\mapsto \pi^*t \in \Sym^d C\) of Remark~\ref{rmk:HilbFunctor} satisfies the conditions from~\eqref{cond:P1inSymd}.

        Now assume~\eqref{cond:P1inSymd}.  Since abelian varieties contain no rational curves, the composition \(\PP^1\to \Sym^d C\to \Pic^d_C\) must be constant.  In particular, the image of \(\PP^1\to \Sym^d C\), which is birational to \(\PP^1\) and contains \(x\), must map to \([x]\).  Thus the fiber of \([x]\) in the Abel--Jacobi map \(\Sym^d C \to \Pic^d_C\) is positive dimensional and so \(\hh^0([x])\geq 2\), as claimed in~\eqref{cond:linearsystem}.

        Assume~\eqref{cond:linearsystem}.  
        Since \(\hh^0([x])\geq 2\), there exists a nonconstant function \(\pi\) whose poles are contained in \(x\).  Since \(x\) is irreducible and \(\pi\) is nonconstant, the poles of \(\pi\) must be equal to \(x\), a degree~\(d\) divisor.   Thus, \(\pi\) gives a degree~\(d\) morphism \(C\to \PP^1\) that sends \(x\) to \(\infty\in \PP^1(k)\), and so \(x\) is \(\PP^1\)-parameterized and \(|x|\) is basepoint free. This yields~\eqref{cond:P1param} and~\eqref{cond:bpf}, respectively.

        Finally, assume~\eqref{cond:bpf}. Then there exists an effective divisor linearly equivalent to \(x\) whose support is not equal to \(x\), so the vector space \(H^0(C, x)\) contains nonconstant functions.  Since \(x\) is effective, \(H^0(C, x)\) also contains the constant functions.  Thus \(h^0([x]) \geq 2\), giving~\eqref{cond:linearsystem}.
    \end{proof}

    Lemma~\ref{lem:P1paramDefineorem} assumes the existence of a closed point of degree~\(d\) and gives criteria for it to be \(\PP^1\)-parameterized.  Similar arguments combined with a corollary of Hilbert's Irreducibility Theorem (Proposition~\ref{prop:hilbert-curve}) show that the existence of basepoint free linear systems imply the existence of some \(\PP^1\)-parameterized point.
    \begin{lemma}\label{lem:bpfGivesP1param}
        Let \(C\) be a nice curve over a number field \(k\) and fix a positive integer \(d\).  Then the following are equivalent.
        \begin{enumerate}
            \item There exists an effective degree~\(d\) divisor \(D\) with 
            \(|D|\) basepoint free.\label{part:bpf}
            \item There exists a degree~\(d\) morphism \(\phi\colon C \to \PP^1\).\label{part:Degreedmorphism}
            \item There are infinitely many degree~\(d\) \(\PP^1\)-parameterized points.\label{part:P1paramInfinite}
            \item There exists a degree~\(d\) \(\PP^1\)-parameterized point.\label{part:P1paramExistence}
        \end{enumerate}
    \end{lemma}
    \begin{remark}\label{rem:gon_P1param}
        Lemma~\ref{lem:bpfGivesP1param} implies that the gonality of a curve is the smallest degree of a \(\PP^1\)-parameterized point.
    \end{remark}

    \begin{proof}
    Lemma~\ref{lem:projection} gives  \eqref{part:bpf} \(\Rightarrow\) \eqref{part:Degreedmorphism}.
        Now assume~\eqref{part:Degreedmorphism}.  Then by Proposition~\ref{prop:hilbert-curve}, there exists a Zariski dense (hence infinite) set of \(t\in \PP^1(k)\) such that \(C_t\) is a degree~\(d\) point, and each such point is necessarily \(\PP^1\)-parameterized; this gives \eqref{part:P1paramInfinite}.  The implication \eqref{part:P1paramInfinite}~\(\Rightarrow\)~\eqref{part:P1paramExistence} is immediate. 
 Lemma~\ref{lem:P1paramDefineorem} proves that \eqref{part:P1paramExistence}~\(\Rightarrow\)~\eqref{part:bpf}.
    \end{proof}

	\subsection{More constructions of curves with Zariski dense degree~\(d\) points}\label{sec:ConstructionCoversOfEllipticCurves}

        In the previous section, we explored how degree~\(d\) morphisms \(C\to \PP^1\) give rise to a Zariski dense set of degree~\(d\) points. On the other hand, if \(C\) is a degree~\(d\) cover of a positive rank elliptic curve, we saw in Section~\ref{sec:HilbIrred} that there need not be degree~\(d\) points contained in the fibers of this map.  Nevertheless, 
        we may use \(\PP^1\)-parameterized points to show that such a curve must always have Zariski dense degree~\(d\) points.

        \begin{prop}\label{prop:covers_E}
            Let \(C\) be a nice curve over a number field \(k\) and let \(C \to C'\) be a degree~\(d\) cover of a nice genus \(1\) curve \(C'\) with \(\rk \Pic^0_{C'}(k) > 0\). Then 
            \[
                d \ind(C'/k)  \mathbb{Z}_{>0} \subset \dendegs(C/k).
            \] 
            In particular, if \(C'(k)\neq\emptyset\), then \(d\mathbb{Z}_{>0} \subset \dendegs(C/k)\).
        \end{prop}

        \begin{proof} Let \(e \in \ind(C'/k)\mathbb{Z}_{>0}\) be at least \(2\). Then there exists a degree~\(e\) divisor \(D\) on \(C'\) and the Riemann--Roch theorem implies that \(\hh^0(D) \geq 2\). Hence by Lemma~\ref{lem:bpfGivesP1param} 
        \(C'\) can be realized as a degree~\(e\) cover of \(\PP^1\). By composing, \(C\) can be realized as a \(de\) cover of \(\PP^1\), so an application of Proposition~\ref{prop:hilbert-curve} implies that \(de\in \dendegs(C/k)\).   

    Thus it remains to show that if \(\ind(C'/k) = 1\) then \(d\in \dendegs(C/k)\). If \(\ind(C'/k) = 1\), then the Riemann--Roch theorem implies that \(C'(k)\neq\emptyset\) so \(C'\simeq \Pic^0_{C'}\) is a positive rank elliptic curve. If \(C\to C'\) does not factor through any nontrivial connected \'etale covers, then Proposition~\ref{prop:Hilbert_E} gives the desired result. 
    Otherwise, let \(C\to C''\xrightarrow{\varphi} C'\)
    be a factorization such that \(C''\) is connected, \(\varphi\) is \'etale, and \(C \to C''\) does not factor through a nontrivial connected \'etale cover.
    We may apply our argument thus far to \(C\to C''\) and conclude that \(\deg(C\to C'')\ind(C''/k)\mathbb{Z}_{>0}\subset \dendegs(C/k)\). 
    
    The fiber of \(\varphi\) over any rational point of \(C'\) decomposes as a sum of closed points of \(C''\), and the index of \(C''\) divides the greatest common divisor of the degrees of these points.
     Since these degrees sum to the degree of \(\varphi\), the index of \(C''\) divides \(\deg(C'' \to C')\), so
    \[
        d = \deg(C \to C'')\deg(C''\to C) \in \deg(C\to C'')\ind(C''/k)\mathbb{Z}_{>0} \subset \dendegs(C/k).\qedhere
    \]
\end{proof}

        One may wonder whether degree~\(d\) morphisms to \(\PP^1\) or elliptic curves of positive rank are the only construction of infinitely many degree~\(d\) points. Harris and Silverman proved that if \(C\) is a nice curve over a number field \(k\) with infinitely many degree~\(2\) points, then \(C\) must be a double cover of a curve of genus \(0\) or \(1\)~\cite{HS-Degree2}.  Abramovich and Harris proved a similar result (allowing passage to a finite extension \(k'/k\)) for degree~\(3\) points on all nice curves, and degree~\(4\) points on all nice curves of genus different from 7~\cite{AH-Degree3And4}.\footnote{Kadets and the second author recently strengthened the Harris--Silverman results to show that \(C\) is a double cover of \(\PP^1\) or an elliptic curve of positive rank~\cite{Kadets--Vogt}*{Theorem 1.2(1)}.  The analogous  strengthening of the Abramovich--Harris result can only fail in genus 3 or 4, see~\cite{Kadets--Vogt}*{Theorem 1.2(2)} .}  However, Debarre and Fahlaoui constructed an example showing that, counter to Abramovich and Harris's expectation, this is false in general: there exists a nice genus \(7\) curve with infinitely many degree~\(4\) points and \emph{no} degree~\(4\) map to \(\PP^1\) or a genus \(1\) curve (even geometrically)~\cite{DF-CounterEx}.

		In essence, the reason that such a curve is possible is the fact that Lemma~\ref{lem:P1paramDefineorem} has no complete analog where \(\PP^1\) is replaced by a genus \(1\) curve. Indeed, the existence of a degree~\(d\) morphism \(f\colon C\to E\) implies the existence of an embedding \(E\to \Sym^d_C, \; P\mapsto f^*P\) as in Remark~\ref{rmk:HilbFunctor}.  But the converse fails.  An embedding \(E\to \Sym^d_C\) does not necessarily imply the existence of a degree~\(d\) morphism \(C\to E\), as the construction of Debarre and Fahlaoui shows.  (If there is an embedding \(E \to \Sym^d_C\) such that the composition \(E \to \Sym^d_C \to \Pic^d_C\) is nonconstant, then there is \emph{some} finite morphism \(C \to E\), but not necessarily of degree~\(d\).)

	\subsection{AV-parameterized points}\label{sec:av}  
        Debarre and Fahlaoui's example shows that we need a more expansive interpretation of what it means for a subset of degree~\(d\) points to be ``parameterized by a positive rank elliptic curve''.  Additionally, Faltings's Subvarieties of AV Theorem suggests that we should also consider parameterizations by positive rank higher dimensional abelian varieties. Taking these as inspiration, we make the following definition (part of which first appeared in~\cite{BELOV}).
		\begin{defn}\label{def:AVparam}
			Let \(C\) be a nice curve over a number field \(k\).
						
			A degree~\(d\) closed point \(x\in C\) is \defi{AV-parameterized} if there exists a positive rank abelian subvariety \(A\subset \Pic^0_C\) such that \([x] + A\subset W_C^d = \im (\Sym^d_C \to \Pic^d_C)\) (c.f. Section~\ref{sec:AJ}).

			A closed point \(x\in C\) is \defi{AV-isolated} if it is not AV-parameterized.
		\end{defn}

        \begin{remark}\label{rem:Zariski_dense}
            Since the Zariski closure \(\overline{A(k)}\) of \(A(k)\) in \(A\) is an abelian subvariety, up to replacing \(A\) with \(\overline{A(k)}\) we may assume that \(A\) has dense \(k\)-points.
        \end{remark}

        The definition of an AV-parameterized point may seem unsatisfying on a first reading. Indeed, the condition that the translate \([x] + A\) is contained in \(W^d_C\) does not invoke the feeling of a ``parameterization'', particularly when compared to the definition of a \(\PP^1\)-parameterized point. However, as we will see in this section, there are many ways that AV-parameterized points behave analogously to \(\PP^1\)-parameterized points, which justifies the terminology.

        In Lemma~\ref{lem:P1paramDefineorem}, we saw that a point \(x\) is \(\PP^1\)-parameterized if and only if there is a \(\PP^1\subset \Sym^d_C\) that contains \(x\).
        In Definition~\ref{def:AVparam}, the geometric object \([x] + A\) parameterizing the AV-parameterized degree~\(d\) points is contained in \(W^d_C\subset\Pic^d_C\) (as opposed to \(\Sym^d_C\)).  However, if \(x\) is AV-parameterized \emph{and} \(\PP^1\)-isolated, then the following lemma shows that a nonempty open in \(A\) is contained in \(\Sym^d_C\).

        \begin{lemma}
    Let \(x\in C\) be an AV-parameterized \(\PP^1\)-isolated degree~\(d\) point. There exists a positive rank abelian variety \(A\) and a rational map \(A\dasharrow \Sym^d_C\) whose image contains \(x\).
\end{lemma}
\begin{proof}
    Since \(x\) is AV-parameterized, there exists a positive rank abelian subvariety \(A\subset \Pic^0_C\) such that \([x] + A\subset W_C^d\). Furthermore, since \(x\) is \(\PP^1\)-isolated, there exists an open set \(U\subset W_C^d\) containing \(x\) over which \(\Sym^d_C\to \Pic^d_C\) is one-to-one. Thus, we have an inverse morphism \(U\cap A\to \Sym^d_C\) as desired.
\end{proof}

\begin{remark}
	Since projective space of any dimension contains \(\PP^1\), any ``\(\PP^n\)-parameterized point'' (under the natural generalization of Lemma~\ref{lem:P1paramDefineorem}\eqref{cond:P1inSymd}) would also be a \(\PP^1\)-parameterized point (cf.~Lemma~\ref{lem:projection}).  However, that is not the case for abelian varieties, since there exist simple abelian varieties of dimension at least~\(2\). For example, any nice curve of genus \(2\) whose Jacobian \(\Pic^0_C\) is simple with \(\rk(\Pic^0_C(k)) \geq 1\) will have \(\PP^1\)-isolated, AV-parameterized points of degree~\(2\) that are \emph{not} parameterized by elliptic curves.  This is consistent with the Harris--Silverman result mentioned above because the canonical linear system is a degree~\(2\) morphism \(C \to \PP^1_k\).  In other words, the results of Harris--Silverman and Abramovich--Harris yield (after possibly passing to a finite extension) a geometric construction that explains an infinite subset
    of the set of degree~\(2\) and \(3\) points, but not necessarily
    \emph{all but finitely many} of the degree~\(2\) or \(3\) points. In Section~\ref{sec:SingleSource} we show that if the genus of \(C\) is large enough, a single geometric construction \emph{does} explain all but finitely many degree~\(d\) points.
\end{remark}

    In Lemma~\ref{lem:bpfGivesP1param}, we saw that the existence of a single \(\PP^1\)-parameterized point gives rise to infinitely many \(\PP^1\)-parameterized points of the same degree. Bourdon, Ejder, Liu, Odumodu, and the first author observed that repeated applications of Faltings's Subvarieties of AV Theorem implies a similar result for AV-parameterized \(\PP^1\)-isolated points.
    \begin{prop}[\cite{BELOV}*{Proof of Theorem 4.2}; relies on~\cite{Faltings-GeneralLang}]\label{prop:P1isolAVparam}
		Let \(C\) be a nice curve over a number field \(k\). Let \(x\in C\) be a degree~\(d\) point that is \(\AV\)-parameterized and \(\PP^1\)-isolated. Let \(A\subset  \Pic^0_C\) be a positive rank abelian variety such that \([x] + A\subset  W_C^d\).  Then there exists a finite index subgroup \(H< A(k)\) such that every element of \([x] + H\) is represented by a degree~\(d\) point.  In particular, there are infinitely many  degree~\(d\) points on \(C\).
	\end{prop}
    While a point \(x\) that is \(\PP^1\)- \emph{and} AV-parameterized also implies the existence of infinitely many \(\PP^1\)-parameterized points of the same degree by Lemma~\ref{lem:bpfGivesP1param},\footnote{{The proof of Lemma~\ref{lem:bpfGivesP1param} can be used to prove the stronger statement that the existence of a point that is \(\PP^1\)- and AV-parameterized implies the existence of infinitely many points that are \(\PP^1\)- and AV-parameterized.}} the more detailed claim given in Proposition~\ref{prop:P1isolAVparam} can fail for such points, as pointed out to us by James Rawson. In other words, the \(\PP^1\)-isolated assumption is necessary; see Example~\ref{ex:rawson} below. 

    Before giving the proof of Proposition~\ref{prop:P1isolAVparam}, we  note that, conversely, AV-isolated points satisfy the following finiteness property.
    
        \begin{cor}
        [Corollary of Faltings's Subvarieties of AV Theorem]
        \label{cor:FiniteImAVisol}
            Fix an integer \(d\geq 1\).  The image of the degree~\(d\) AV-isolated points under the Abel--Jacobi map is finite.
        \end{cor}
        \begin{proof}
        Degree~\(d\) points map to \(W^d_C(k)\) under the Abel-Jacobi map. 
        By Faltings's Subvarieties of AV Theorem, there exist finitely many points \(D_1, \dots D_r\in W^d_C(k)\) and finitely many abelian subvarieties \(A_1, \dots A_r\subset  \Pic^0_C\) such that \(D_i + A_i \subset  W^d_C\) for all \(i\) and
        \[
            W^d_C(k) = \bigcup_{i=1}^r D_i + A_i(k).
        \]
        By definition, degree~\(d\) AV-isolated points must map to a translate \(D_i + A_i(k)\), where \(A_i\) has rank~\(0\). Since a finite union of translates of rank~\(0\) abelian subvarieties has finitely many points, there are only finitely many images of degree~\(d\) AV-isolated points.
        \end{proof}

	\begin{proof}[Proof of Proposition~\ref{prop:P1isolAVparam}]
		By Lemma~\ref{lem:FiniteIndexEffective}, there is a finite index subgroup \(B <  A(k)\) such that every point in \([x]+B\) is represented by an effective \(k\)-rational divisor.  By Lemma~\ref{lem:image_summation_reducible}, all points of \([x] + B\) that are in the complement of \(\im \left(+ \colon W_C^{\mathbf{d}}(k)\to W_C^d(k)\right)\) for all nontrivial partitions \(\mathbf{d}\vdash d\) give integral divisors, i.e., degree~\(d\) points. Thus it suffices to prove that there is a finite index subgroup \(H<  B\) such that \([x] + H\) is disjoint from \(\im \left(+ \colon W_C^{\mathbf{d}}(k)\to W_C^d(k)\right)\) for all nontrivial partitions \(\mathbf{d}\vdash d\). By translating by \(-[x]\), this is equivalent to showing the existence of a finite index subgroup \(H< B\) that is disjoint from \(\im \left(W^{\mathbf{d}}(k)\xrightarrow{+} W_C^d(k) \xrightarrow{-[x]} \Pic^0_C(k)\right)\). Since \(x\) is \(\PP^1\)-isolated, \(x\notin \im \left(+ \colon W_C^{\mathbf{d}}(k)\to W_C^d(k)\right)\) for all nontrivial partitions \(\mathbf{d}\vdash d\) (see Lemma~\ref{lem:P1isolatedIrreducible} below), so \(0\notin\im \left(W^{\mathbf{d}}(k)\xrightarrow{+} W_C^d(k) \xrightarrow{-[x]} \Pic^0_C(k)\right)\).

        Fix a nontrivial partition \(\mathbf{d}=(d_i)_{i=1}^r\vdash d\). By Faltings's Subvarieties of AV Theorem, the set of rational points \(W_C^{\mathbf{d}}(k)\) is a finite union of translates of \(k\)-points on abelian subvarieties, so the image \(\im \left(W^{\mathbf{d}}(k)\xrightarrow{+} W_C^d(k) \xrightarrow{-[x]} \Pic^0_C(k)\right)\) is also a finite union of translates of \(k\)-points on abelian subvarieties. In other words, there is a finite collection of subgroups \(G_i <  A(k)\) (necessarily finitely generated) and divisors \(D_i\) such that 
	\[
		\bigcup_{\substack{\mathbf{d}\vdash d\\\mathbf{d} \neq (d)}}\im\left(W^{\mathbf{d}}(k)\xrightarrow{+} W_C^d(k) \xrightarrow{-[x]} \Pic^0_C(k)\right) = \bigcup_{i=1}^s D_i + G_i	
	\]

        Fix \(1\leq i \leq s\). The intersection  \(B\cap (D_i + G_i)\) is either empty or it contains some \(D_i' \in D_i  + G_i\).  Since \((D_i + G_i)- D_i' = G_i\), we see that \(B\cap (D_i + G_i) = D_i' + B\cap G_i\).  Up to replacing \(D_i\) by \(D_i' \in D_i + G_i\), we may therefore assume that \(B\cap (D_i + G_i) = \emptyset\) or \(B\cap (D_i + G_i) =\)
        \(D_i + B\cap G_i\). In particular, after possibly reordering the \((D_i, G_i)\), we have:
        \[
            B\cap \left[\bigcup_{\substack{\mathbf{d}\vdash d\\\mathbf{d} \neq (d)}}\im\left(W^{\mathbf{d}}(k)\xrightarrow{+} W_C^d(k) \xrightarrow{-[x]} \Pic^0_C(k)\right)\right] = \bigcup_{i=1}^{s'} D_i + B\cap G_i \subset  B\smallsetminus\{0\}.
        \]
        Recall that we wish to show that the finitely generated group \(B\) contains a finite index subgroup \(H< B\) such that the union of cosets \(\bigcup_{i=1}^{s'} D_i + B\cap G_i\) misses all of \(H\). It is a corollary of the fundamental theorem of finitely generated abelian groups (see Lemma~\ref{lem:CorOfFGAG} below) that any finite union of cosets in a finitely generated abelian group either covers the group completely or is disjoint from a coset of a finite index subgroup. Since the union of cosets \(\bigcup_{i=1}^{s'} D_i + B\cap G_i\) already omits \(0\), it must avoid a finite index subgroup \(H\).
\end{proof}

	\begin{lemma}\label{lem:P1isolatedIrreducible}
		If \(x\in C\) is a point of degree~\(d\) that is \(\PP^1\)-isolated, then \([x]\) is not in the image of \(W_C^e{(k)}\times W_C^{d-e}{(k)}\) for any \(0<e<d\).
	\end{lemma}
	\begin{proof}
        Consider some \(0 < e < d\).
        By assumption, \(x\) is a degree~\(d\) point, and hence 
        by Lemma~\ref{lem:image_summation_reducible} it is not in the image of the summation \(\sigma_{(e, d-e)} \colon \Sym^{(e, d-e)}_C(k) \to \Sym^d_C(k)\).
          \begin{center}
        \begin{tikzcd}
            \Sym^{(e, d-e)}_C(k) \arrow[d, "\sigma_{(e, d-e)}"] \arrow[r] & W_C^{(e, d-e)}(k)  \arrow[d]  \\
            \Sym^d_C(k) \arrow[r]& W_C^d(k) 
        \end{tikzcd}
    \end{center}
    Since \(x\in C\) is \(\PP^1\)-isolated,  Lemma~\ref{lem:P1paramDefineorem} implies that \(\hh^0([x]) = 1\).  Thus, 
        \(x\) is the unique preimage of the Abel--Jacobi map \(\Sym^d_C(k) \to W_C^d(k)\) over \([x]\).  Further, any point in \(W_C^{(e,d-e)}(k)\) mapping to \([x]\) would also have a unique preimage under the Abel--Jacobi map and thus  would be in the image of \(\Sym^{(e, d-e)}_C(k)\) by Lemma~\ref{lem:BrSev}.  Hence no such point exists.
	\end{proof}

\begin{lemma}\label{lem:CorOfFGAG}
    Let \(B\) be a finitely generated abelian group. Let \(I\) be a finite set and, for each \(i \in I\), let \(H_i <  B\) be a subgroup and let \(y_i\in B\). If there exists an \(x\in B\) that is not in \(y_i + H_i\) for any \(i \in I\), then there is a finite index subgroup \(H <  B\) such that \((x + H) \cap (y_i + H_i) = \emptyset\) for all \(i\in I\). In other words,
    \[
        B\setminus \left(\bigcup_{i\in I} y_i + H_i \right) = \emptyset \quad \textup{or} \quad
        B\setminus \left(\bigcup_{i\in I} y_i + H_i \right) \supset x + H \;\textup{for some \(H < B\) of finite index}.
    \]
\end{lemma}
\begin{proof}
    Since \(x\notin y_i + H_i\), the class \(z_i := x -y_i+ H_i\in B/H_i\) is nontrivial for all \(i\). By the fundamental theorem of finitely generated abelian groups, \(\bigcap_{m\geq 1}m(B/H_i)\) is trivial, thus there is some positive integer \(m_i\) such that \(z_i\not\in m_i(B/H_i)\).  Let \(H_i' < B\) be the preimage of \(m_i(B/H_i)\) in \(B\); then \(H_i'< B\) is of finite index, \(x \not\in y_i + H_i'\), and \(H_i <  H_i'\).
    Set \(H := \bigcap_i H_i'\). Since each \(H_i'\) is of finite index in \(B\), the intersection \(H\) is of finite index in \(B\).  Since \(x \not\in y_i + H_i'\) and \(H <  H_i'\) we have \((x + H )\cap(y_i  + H_i') = \emptyset\) for all \(i\).  Thus \((x + H )\cap(y_i  + H_i) = \emptyset\) since \(y_i + H_i \subset  y_i+ H_i'\).
\end{proof}

    \begin{example}\label{ex:rawson}
        This example is due to James Rawson. Let \(C\) be a nice genus \(4\) hyperelliptic curve with a degree~\(3\) map \(\pi \colon C \to E\) to a positive rank elliptic curve \(E\) and assume that there exists a rational Weierstrass point \(p \in C(k)\) that is also a point of total ramification for \(\pi\). Fix the isomorphism \(E\simeq \Pic^0(E)\) that identifies \(\pi(p)\in E\) with the zero divisor.  (For example, we may take \(C: y^2 = u^9- 6u^6 - 4u^3 + 40\), which maps to the elliptic curve \(E:y^2 = x^3 -16x + 16\) by the map \((u,y) \mapsto (u^3 - 2, y)\). The Weierstrass point \(p\in C(\Q)\) that lies over \(\infty\in \PP^1\) is totally ramified under this morphism and \(E\) has Mordell--Weil rank 1~\cite{lmfdb}*{\href{https://www.lmfdb.org/EllipticCurve/Q/37/a/1}{37.a1}}.)

        Pullback induces an embedding \(E\simeq\Pic^0_E \hookrightarrow \Pic^0_C\), which in turn gives \([4p] + E \subset W_C^4\) as the locus of divisor classes of the form \([p] + \pi^{-1}(y)\) for some \(y\in E\). We will show that there exists a degree~\(4\) point \(x\in C\) such that \([x] = [4p]\), but that all points of  \([x] + (E(k)\smallsetminus\{O\})\) are represented only by non-integral divisors. In particular, \(x\) is AV-parameterized, \(\PP^1\)-parameterized, and the conclusion of Proposition~\ref{prop:P1isolAVparam} fails. 
        
        Since \(p\) is a Weierstrass point (hence \(4p\) is twice a hyperelliptic divisor), the linear system \(|4p|\) is basepoint free of dimension \(2\).  On the other hand, by the adjunction formula (c.f.~Section~\ref{sec:single_source_key}) any degree~\(4\) map \(C \to \PP^1\) factors through the hyperelliptic map.  Hence \([4p]\) is the \emph{unique} point in \(W_C^4\) whose complete linear system is basepoint-free.
       
       Since \(|4p|\) is basepoint-free, by Lemma~\ref{lem:bpfGivesP1param} there exists a degree~\(4\) point \(x \in C\) such that \([4p] = [x]\). By definition, every point in \([x] + E(k)\) is represented by a reducible divisor \(p + \pi^{-1}(y)\) for \(y \in E(k)\).  If \(y \neq \pi(p)\), the linear system \(|p + \pi^{-1}(y)|\neq |4p|\) has a basepoint by the argument above and thus every effective divisor \(D\in |p + \pi^{-1}(y)|\) must contain at least one of the \(k\)-rational effective divisors \(\{p\}\cup \textup{Supp}(\pi^{-1}(y))\) of degree at most \(3\) in its support. In particular, \(D\) is never integral.  Thus \([x]\) is the unique element of \([x] + E(k)\) represented by a degree~\(4\) point, and no finite index subgroup \(H <  E(k)\) as in Proposition~\ref{prop:P1isolAVparam} exists.\hfill\(\righthalfcup\)
    \end{example}

  \subsubsection{Comparison with the Ueno locus}
            Let \(A/F\) be an abelian variety and let \(X\subset  A\) be a subvariety.  The \defi{Ueno locus} of \(X\) is the union of the positive dimensional (geometric) cosets contained in \(X\), i.e., 
            \[
            \Ueno(X) \colonequals \bigcup_{\substack{B\subset  A_{\bar{F}}, x\in X_{\bar{F}}\\ \dim B >0\\x + B\subset  X}} x + B,
            \]
            where \(B\subset  A_{\bar{F}}\) denotes an abelian subvariety of \(A_{\bar{F}}\). Since \(X\) is defined over \(F\), if a geometric coset \(x+B\) is contained in \(X\) then any Galois conjugate of \(x + B\) is also contained in \(X\). Hence, the Ueno locus of \(X\) is  defined over \(F\) as well.
            The Ueno locus is equivariant under translation, i.e., \(\Ueno(X+a)=\Ueno(X)+a\) for any \(a\in A\). If \(X\) is a translate of an abelian variety then the Ueno locus of \(X\) is all of \(X\). 
            In general, results of Ueno~\cite{Ueno}*{Theorem 3.10} and Kawamata~\cite{Kawamata}*{Theorem 4} together imply that the Ueno locus is Zariski closed (see~\cite{Lang-NTIII}*{Chapter 1, Section 6} for more details on how this implication follows from Ueno and Kawamata's work).
            
            \begin{defn}\label{def:Ueno}
            Let \(C\) be a curve. We say that a closed point \(x\in C\) is \defi{Ueno} if its divisor class \([x]\) lies in the Ueno locus of \( W_C^{\deg(x)}\subset  \Pic^{\deg(x)}_C\), and \defi{non-Ueno} otherwise.  
            \end{defn}
             Thus, every closed point of degree at least~\(g\) on a curve of genus \(g\) is Ueno.
            
            Since any positive rank abelian variety is also positive dimensional, the definitions of AV-parameterized points and the Ueno locus immediately imply the following.
            \begin{lemma}
                Let \(x\in C\) be a closed point.  If \(x\) is AV-parameterized then \(x\) is Ueno. \qed
            \end{lemma}

            However, the converse does not hold.  There can exist Ueno AV-isolated points, e.g., any degree at least \(g\) point on a genus \(g\) curve whose Jacobian has rank \(0\). The following example shows that Ueno AV-isolated points can also have degree less than the genus.
            \begin{example}
                Consider the genus \(3\) curve \(C\) defined by \(y^2 = t^8 + 8t^6 + 22t^4 + 25t^2 + 10\).  We have the following degree \(2\) morphisms:
                \[
                \begin{tikzcd}
                & C
                \arrow[rd, "{(t, y)\mapsto (t^2, y)}"] 
                \arrow[ld, "{(t,y)\mapsto t}", swap] & \\
                \PP^1 & & E\colon y^2 = x^4 +8x^3 + 22x^2 + 25x + 10
                \end{tikzcd}
                \]
                The genus \(1\) curve \(E\) has exactly \(6\) rational points,\footnote{This statement and all other computational claims in this example are verified in \texttt{Magma} code~\cite{magma} available on \texttt{Github}~\cite{VV-code}.} including \((-3/2, \pm 1/4)\).  The preimages of \((-3/2, \pm 1/4)\) are quadratic points on \(C\) that map to quadratic points on \(\PP^1\).  
                The Jacobian of \(C\) has rank \(0\), so every closed point is AV-isolated. In particular, the quadratic points on \(C\) that map to \((-3/2, \pm 1/4)\in E\) are isolated.  However, the map \(C\to E\) gives a morphism \(E\to \Pic^2_C\), and by definition, the fibers of \(C\to E\) are contained in the image of \(E\).  Therefore, these quadratic points are Ueno AV-isolated points.\hfill\(\righthalfcup\)
            \end{example} 
            Ueno AV-isolated points should be thought of as fundamentally different from either non-Ueno AV-isolated points or (Ueno) AV-parameterized points, since they are not stable under base extension. For example, an Ueno AV-isolated point \(x\) corresponding to a rank \(0\) positive-dimensional abelian subvariety \(B \subset  \Pic^0_C\) will become (Ueno) AV-parameterized after a finite extension \(k'/k\) that is linearly disjoint from \(\kk(P)\) and for which \(\rk \Pic^0_C(k') > 0\).  On the other hand, a non-Ueno AV-isolated point \(x\) remains non-Ueno AV-isolated under any finite extension linearly disjoint from \(\kk(x)\). In the figure below, the grey polka-dot semicircle region represents the AV-isolated points that will become AV-parameterized after a finite extension.
\begin{figure}[hbt!]
                \centering
                \includegraphics[width=0.4\linewidth]{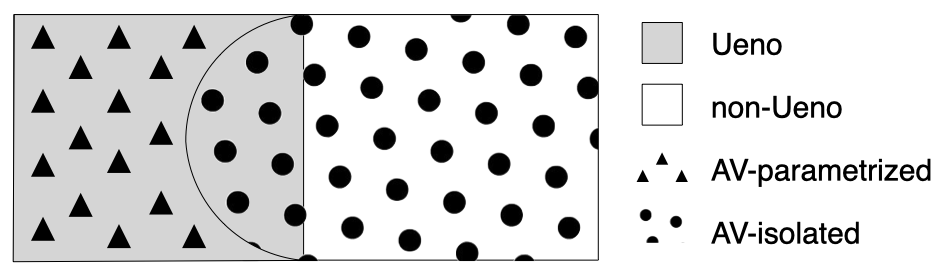}
            \end{figure}

\subsection{Finiteness of isolated points}\label{sec:Dichotomy}

	\begin{defn}
		Let \(C\) be a nice curve over a number field \(k\).  
		
		A closed point \(x\in C\) is \defi{parameterized} if it is \(\PP^1\)-parameterized \emph{or} AV-parameterized.

		A closed point \(x\in C\) is \defi{isolated} if it is both \(\PP^1\)-isolated \emph{and} AV-isolated.
	\end{defn}
	\begin{remark}Note that parameterized and isolated are properties of \emph{algebraic} points, and curves of any genus always have infinitely many algebraic points.
	\end{remark}

	The definition of parameterized points provides additional structure (albeit not necessarily in an explicit way) that we can use to study these points.  Indeed, a \(\PP^1\)-parameterized point of degree~\(d\) is witnessed by a degree~\(d\) morphism to \(\PP^1\).  An AV-parameterized degree~\(d\) point \(x\) is witnessed by an abelian subvariety \(A\subset  \Pic^0_C\) of positive rank such that the translate \(x + A\) is contained in \(\im (\Sym^d_C\to \Pic^d_C)\).  Hence, the parameterized points should be thought of as the ones with a \emph{geometric} reason for existence.  The isolated points are mysterious -- there is no good geometric reason for their existence.  
    	\begin{theorem}[\cite{Faltings-GeneralLang}\(+\varepsilon\); \cite{BELOV}*{Theorem 4.2}]\label{thm:InfinitudeDegreed}
		Let \(C\) be a nice curve over a number field. 
        \begin{enumerate}
            \item There are finitely many isolated points on \(C\) (regardless of degree).\label{part:IsolatedFinite}
            \item There are infinitely many degree~\(d\) points if and only if there exists a degree~\(d\) parameterized point.
        \end{enumerate}
	\end{theorem}
    \begin{proof}\hfill
    \begin{enumerate}
        \item This follows from Corollaries~\ref{cor:HighDegreeP1Param}, \ref{cor:BijP1isolated}, and~\ref{cor:FiniteImAVisol}.
        \item The forwards direction follows from  part~\eqref{part:IsolatedFinite}.
		The backwards direction follows from Lemma~\ref{lem:bpfGivesP1param}
  and Proposition~\ref{prop:P1isolAVparam}.\qedhere
        \end{enumerate}
    \end{proof}

	Theorem~\ref{thm:InfinitudeDegreed}\eqref{part:IsolatedFinite} has a particularly nice consequence for curves \(C/k\)  whose Jacobian has rank \(0\).  In this case, \(C\) has no AV-parameterized points, so, by Theorem~\ref{thm:InfinitudeDegreed}\eqref{part:IsolatedFinite} all but finitely many closed points of \(C\) are \(\PP^1\)-parameterized. The definition of \(\PP^1\)-parameterized points then gives:
	\begin{cor}[Specialization of Theorem~\ref{thm:InfinitudeDegreed}\eqref{part:IsolatedFinite}]
		Let \(C\) be a curve over a number field \(k\) whose Jacobian has rank \(0\). Then there exists a proper closed subscheme \(Z\subsetneq C\) such that \emph{all closed points} \(x\in C\setminus Z\) appear as a fiber of some morphism to \(\PP^1\).
	\end{cor}

\section{The density degree set
for curves over number fields}\label{sec:DensityDegrees}

Since a proper Zariski closed subset of a curve consists of a \textit{finite} union of closed points, Theorem~\ref{thm:InfinitudeDegreed} connects the general notion of the density degree set \(\dendegs(C/k)\) of a curve \(C\) defined over a number field \(k\) with degrees of parameterized points.

\begin{cor}[Corollary of Theorem~\ref{thm:InfinitudeDegreed}]\label{cor:dendeg_parameterized}
Let \(C\) be a nice curve over a number field \(k\).  An integer \(d\) is in \(\dendegs(C/k)\) if and only if there exists a parameterized point of degree~\(d\).
\end{cor}
Thus, Corollary~\ref{cor:dendeg_parameterized} implies that \(\dendegs(C/k)\) is the union of the following two sets.

\begin{defn}\hfill
    \begin{enumerate}
        \item \(\dendegs_{\PP^1}(C/k) \: \colonequals \{\deg(x) :\text{\(x\) is a \(\PP^1\)-parameterized point}\}\).\footnote{By Lemma~\ref{lem:P1paramDefineorem}, \(\dendegs_{\PP^1}(C/k)\) is the same as the L\"uroth semigroup of \(C/k\).}
        \item \(\dendegs_{\AV}(C/k) \colonequals \{\deg(x) :\text{\(x\) is an AV-parameterized point}\}\).
    \end{enumerate}
\end{defn}

In this section, we determine various constraints on the structure of these sets.

\subsection{Asymptotics of the density degree set}
When the degree is large enough, the behavior of the degrees of \(\PP^1\)-parameterized points (and thus the integers in the density degree set) is completely determined by the index of the curve.

\begin{prop}\label{prop:DenDegLarge}
Let \(C\) be a nice curve over a number field \(k\) and let \(g\) denote the genus of \(C\).
If  \(d \geq \max(2g, 1)\), then the following are equivalent:
\begin{enumerate}
\item\label{gonality} \(d \in \dendegs_{\PP^1}(C/k)\),
\item\label{density} \(d \in \dendegs(C/k)\),
\item\label{degreeset} \(d \in \calD(C/k)\).
\item\label{index} \(d\) is a multiple of the index of \(C\),
\end{enumerate}
\end{prop}
\begin{proof}
The implication \eqref{gonality} \(\Rightarrow\) \eqref{density} is Corollary~\ref{cor:dendeg_parameterized}. The implication \eqref{density} \(\Rightarrow\) \eqref{degreeset} is immediate. The implication \eqref{degreeset} \(\Rightarrow\) \eqref{index} follows from the definition \(\ind(C/k) \colonequals \gcd\calD(C/k)\).

It remains to prove \eqref{index} \(\Rightarrow\)\eqref{gonality}. If \(d\) is a multiple of the index, then there exists a linear combination of closed points, i.e., a \(k\)-rational divisor \(D\), of degree~\(d\). Since \(d \geq \max(2g,1)\), the Riemann--Roch Theorem implies that \(h^0(D) \geq \max(g+1, 2) \geq 2\) and that \(|D|\) is basepoint free. Hence, \(C\) has a \(\PP^1\)-parameterized point of degree~\(d\) by Lemma~\ref{lem:bpfGivesP1param}.  
\end{proof}

\begin{remark}The proof of Proposition~\ref{prop:DenDegLarge} crucially uses Hilbert's irreducibility theorem, and hence that the ground field is Hilbertian.  If, instead, the ground field were Henselian, counterexamples to the analogous statement are given in \cite{Creutz--Viray}.
\end{remark}

The assumption \(d \geq \max(2g, 1)\) in Proposition~\ref{prop:DenDegLarge} cannot be weakened to \(d \geq g+1\), which merely guarantees that a linear system has dimension at least~\(1\), without  basepoint freeness.  Compare this to Corollary~\ref{cor:HighDegreeP1Param}, where the fact that \(x\) is a closed point guarantees basepoint freeness.  The following example illustrates this, as well as the fact that \(\Sym^d_C(k)\) being infinite is not sufficient to guarantee that \(d \in \dendegs(C/k)\).

\begin{example}
Let \(C\) be a curve of genus \(2\) for which \(C(k) =\{p\}\) and \(\Pic^0_C(k) = \{\OO_C\}\) (in particular, \(\Pic^0_C\) has rank \(0\) over~\(k\)).  Since \(C(k) \neq \emptyset\), we have \(1 \in \degset(C/k)\); on the other hand, \(1 \notin \dendegs(C/k)\) by Faltings's Curve Theorem.  The complete linear system of the canonical divisor (class) \(K_C\) gives a degree~\(2\) map \(C \to \PP^1_k\), and hence \(2 \in \dendegs(C/k)\).  Since \(\Pic^0_C(k) = \{\OO_C\}\), the only \(k\)-point on \(\Pic^3_C\) is \(K_C + [p]\).  Even though \(h^0(K_C +[p]) = 2\), there is no degree~\(3\) map  \(C \to \PP^1_k\): since \(h^0(K_C) = 2\), the point \(p\) is a basepoint of this linear system.  Hence \(3 \notin \degset(C/k)\) (even though \(3\geq 2 + 1\)), and consequently \(3 \not\in \dendegs(C/k)\).  Combining this with Proposition~\ref{prop:DenDegLarge}, we see that \(\degset(C/k) =  \Z_{>0} \smallsetminus \{3\}\) and \(\dendegs(C/k) = \Z_{>0} \smallsetminus \{1, 3\}\).
Examples of such curves over~\(\Q\) can be found by the following \href{https://www.lmfdb.org/Genus2Curve/Q/?num_rat_pts=1&torsion\_order=1&analytic_rank=0}{search}~ \cite{lmfdb}. \hfill\(\righthalfcup\)
\end{example}

The degrees of AV-parameterized points also have a uniform behavior when \(d\) is large.

\begin{lemma}\label{lem:AVDensityDegreeLarge}
    Let \(C\) be a nice curve over a number field \(k\).  If \(d \geq g \geq 1\), then the following are equivalent:
    \begin{enumerate}
        \item \(d \in \dendegs_{\AV}(C/k)\),
        \item \(d \in \calD(C/k)\) and \(\rk(\Pic^0_C(k)) > 0\).
    \end{enumerate}
\end{lemma}
\begin{proof}
    Since \(W_C^d = \Pic^d_C\) for all \(d \geq g\) by the Riemann--Roch theorem, once there exists a closed point \(x\) of degree~\(d \geq g\), we have \(x + \Pic^0_C \subset  W_C^d\).
\end{proof}

\begin{remark}\label{rem:AVsimple}
    If \(\Pic^0_C\) is a simple abelian variety, then there are no AV-parameterized points of degree~\(1 \leq d \leq g-1\), since \(\dim W_C^d = d < g\) in this range.  Hence the existence of AV-parameterized points is controlled exclusively by \(\calD(C/k)\) and \(\rk(\Pic^0_C(k))\) by Lemma~\ref{lem:AVDensityDegreeLarge}.
\end{remark}

\begin{remark}
    If \(\rk(\Pic^0_C(k)) = 0\), then every \(x \in C\) is AV-isolated and \(\dendegs_{AV}(C/k) = \emptyset\).
\end{remark}

\subsection{Multiplicative structure of the density degree set}
Since \(\dendegs_{\PP^1}(C/k)\) is a semigroup, it
is closed under multiplication by positive integers. 
The set \(\dendegs_{AV}(C/k)\), and hence \(\dendegs(C/k)\), also has this closure property.

\begin{prop}\label{prop:DensdegsClosedUnderScalarMult}
    Let \(C\) be a nice curve defined over a number field \(k\).  If \(n\geq2\) and \(x\in C\) is an AV-parameterized degree~\(d\) point, then there exists a \(\PP^1\)-parameterized degree~\(nd\) point \(z\) that is linearly equivalent to \(n[x]\).  This implies that 
    \[\mathbb{Z}_{\geq 2} \cdot \dendegs_{\AV}(C/k) \subset  \dendegs_{\AV}(C/k)\cap \dendegs_{\PP^1}(C/k),\]
    and so \(\dendegs(C/k)\) is closed under multiplication by \(\mathbb{Z}_{\geq1}\).
\end{prop}

\begin{proof}
    There exists a positive rank abelian variety \(A\subset  \Pic^0_C\) such that \(x + A\subset  W_C^d\).  
    Since all positive multiples of an effective divisor are also effective, we have \([nx] + [n]A \subset  W_C^{nd}\), where \([n]\) is the multiplication by \(n\) map on \(\Pic^0_C\).  Since \([n]A = A\), if we can show that there exists a closed point \(z \in |nx|\), then \(z + A \subset  W_C^{nd}\) and \(nd \in \dendegs_{\AV}(C/k)\).
    
    By Lemma~\ref{lem:FiniteIndexEffective}, there exists a finite index subgroup \(H<  A(k)\) such that every divisor in \(x + H\) is in the image of \(\Sym^d_C(k)\).  Let \(D_0\in H<  A(k)\) be a point of infinite order. Then for all integers \(i\), the divisor class \(x + iD_0\) is represented by some degree~\(d\) \(k\)-rational effective divisor \(y_i\) (we will take \(y_0 = x\)).  Note that since \(D_0\) has infinite order \(y_i\not\sim y_j\) if \(i\neq j\).  Then for any integer \(n\) and any integer \(i\), we have
    \[
        (n-1)y_i + y_{i - in} \;\sim \; (n-1)(x + iD_0) + (x + (i - in)D_0) \; \sim \; nx = ny_0.
    \]
    Thus, for \(n \geq 2\), \([nx]\) has infinitely many effective representatives with disjoint support and so \(\hh^0([nx])\geq 2\) and \([nx]\) is basepoint free.  Hence, by Lemma~\ref{lem:bpfGivesP1param},
    there exists a closed point of degree~\(nd\) in \(|nx|\) and we conclude \(nd \in \dendegs_{\PP^1}(C/k) \cap \dendegs_{AV}(C/k)\).
\end{proof}

The \(n=2\) case of Proposition~\ref{prop:DensdegsClosedUnderScalarMult}, which was proved by Abramovich--Harris and Frey, independently, implies that geometric properties (e.g., the gonality) can prohibit the existence of a parameterized point, giving another example where geometry controls arithmetic.

\begin{cor}[\cites{AH-Degree3And4,Frey}]\label{cor:GeometricConsequences}
    Let \(C\) be a nice curve over a number field \(k\).  If \(d \in \dendegs_{\AV}(C/k)\), then \(\gon(C) \leq 2d\).
\end{cor}

\subsection{Additive structure of the density degree set}
 Given Proposition~\ref{prop:DensdegsClosedUnderScalarMult} and the fact that \(\dendegs_{\PP^1}(C/k)\) is a semigroup, one may ask whether  \(\dendegs(C/k)\) \textit{also} has the structure of an additive semigroup.
  The following example shows that this is not necessarily the case.

\begin{example}\label{ex:DenDegsNotSemigroup}

Let \(E/\Q\) be the elliptic curve defined by \(y^2 = x^3 -x + 1\) whose Mordell--Weil group is free of rank 1~\cite{lmfdb}*{\href{https://www.lmfdb.org/EllipticCurve/Q/92/a/1}{92.a1}}. Let \(\psi\colon \PP^1\to \PP^1\) be the degree~\(3\) map given by \([t_0:t_1]\mapsto [t_0^3 + 7t_0^2t_1 - 9t_0t_1^2 + 7t_1^3: t_0t_1(t_0 - t_1)]\) and let \(\varphi \colon E \to \PP^1\) be the double cover given by \((x,y) \mapsto [x:1]\).  Let \(C\) be the fiber product \(E\times_{\PP^1}\PP^1\) via \(\varphi\) and \(\psi\).  Then \(C\) is a smooth hyperelliptic curve of genus \(5\) that is a degree~\(3\) cover of \(E\).  In particular, \(C\) has a degree~\(2\) \(\PP^1\)-parameterized point and a degree~\(3\) AV-parameterized point. We claim that \(C\) has no degree~\(5\) parameterized points.  The key features for this proof are that 
\(C\) cannot admit a degree~\(3\) or \(5\) map to \(\PP^1\) by the adjunction formula (c.f.~Section~\ref{sec:single_source_key})
and that \(\Pic_C(\Q)/\Pic_E(\Q)\) is generated by rational Weierstrass points of \(C\).

Since there is no degree \(5\) map \(C\to \PP^1\), \(5\notin \delta_{\PP^1}(C/k)\) by Lemma~\ref{lem:bpfGivesP1param}. Thus it remains to prove that \(5\notin \delta_{AV}(C/k)\).  Since \(g=5\), we have \(W_C^5 = \Pic^5_C\), and since \(E\) has rank \(1\), the Jacobian of \(C\) has positive rank. Thus any point in \(\Pic^5_C(\Q)\) that is represented by an \emph{irreducible} divisor will be an AV-parameterized point.  We must therefore prove that \(5 \not\in \calD(C/\Q)\).  In other words, we must show that every point in \(\Pic^5_C(\Q)\) is represented by a \emph{reducible} divisor. (Since there are no \(\PP^1\)-parameterized points of degree~\(5\), every degree~\(5\) linear system has a base point. In particular, if there is a reducible representative of a degree~\(5\) divisor class, then every effective representative of this class will be reducible.)
 
First we determine \(\Pic^5_C(\Q)\). The degree~\(3\) morphism \(C\to E\) induces an embedding \(E\hookrightarrow \Sym^3_C\) where all the points in the image have disjoint support.  In addition, \(C\) has no degree~\(3\) map to \(\PP^1\) because its genus is too large, so composing with the degree~\(3\) Abel--Jacobi map and translating gives an embedding \(\Pic^0_E\hookrightarrow \Pic^0_C\).  A \texttt{Magma} \cite{magma} computation shows that the rank of \(\Pic^0_C(\Q)\) is at most \(1\)~\cite{VV-code}, so \(\left(\Pic^0_C/\Pic^0_E\right)(\Q)\) must be torsion.  This finite torsion group must embed into \(\left(\Pic^0_C/\Pic^0_E\right)(\F_p)\) for any prime \(p\) of good reduction.  By Lang's Theorem, this is isomorphic to  \(\Pic^0_C(\F_p)/\Pic^0_E(\F_p)\)~\cite{Lang-GroupTorsors}*{Theorem 2}, which for \(p=5\) and \(p=13\), has order \(2^3\cdot3\cdot29\) and \(2^2\cdot17\cdot 311\), respectively. Thus, \(\left(\Pic^0_C/\Pic^0_E\right)\) has order at most \(4\).  By construction, \(C\) has \(3\) rational Weierstrass points \(w_0,w_1,w_2\) (contained in \(\psi^{-1}(\infty)\)), which generate an order \(4\) torsion subgroup of \(\Pic^0_C\).  Since \(E\) has a free Mordell--Weil group, these Weierstrass points must generate \(\left(\Pic^0_C/\Pic^0_E\right)(\Q)\) and so \(\Pic^0_C(\Q)\simeq \Pic^0_E(\Q)\times \Z/2\times \Z/2\). 
In other words, \(W_C^5(\Q) = \Pic^5_C(\Q)\) consists of four translates of \(\psi^*\Pic^1_E\) by the following four degree~\(2\) combinations of Weierstrass points:
\[
   [2w_0] = [2w_1] = [2w_2], \;  [w_0 + w_1], 
   \; [w_0 + w_2], \; \textup{and}\; [w_1 + w_2].
\]
In particular, every divisor in these translates has a reducible effective representative.\hfill\(\righthalfcup\)
\end{example}

\subsection{The minimum density degree of a curve over a number field}

If all closed points on \(C\) of degree strictly less than \(\gon(C)\) are AV-isolated, for example, if \(\rk \Pic^0_C(k) = 0\), then 
\(\min(\dendegs(C/k)) = \gon(C)\).
Indeed, by definition, all points on \(C\) of degree strictly less than \(\gon(C)\) are \(\PP^1\)-isolated.  Hence, all points on \(C\) of degree strictly less than \(\gon(C)\) are isolated if and only if they are AV-isolated.  

However, in general, the minimum density degree of a curve over a number field  also accounts for the additional possibility of AV-parameterized points of small degree.  By Corollary~\ref{cor:GeometricConsequences} and Lemma~\ref{lem:bpfGivesP1param}, the minimum density degree is bounded:
\begin{equation}\label{gonbounds}
    \gon(C)/2 \leq \min(\dendegs(C/k)) \leq \gon(C).
\end{equation}
In particular, for \(g \neq 1\), since either \(|K_C|\) or \(|{-K_C}|\) is a basepoint free linear system of degree~\(|2g-2|\), we have the upper bound \(\min(\dendegs(C/k)) \leq |2g-2|\). (If \(g=1\), then \(\min(\dendegs(C/k)) =\ind(C/k)\), which is unbounded~\cite{Sharif-PeriodIndexNumberFields}*{Theorem 1.2}.)

For arbitrary curves,  the bounds in~\eqref{gonbounds} are known to be sharp~\cite{Smith--Vogt}*{Theorem 1.1}.
But in the case when \(\Pic^0_C\) is a simple abelian variety, there is only one case where \(\min(\dendegs(C/k))\) can differ from \(\gon(C)\).

\begin{cor}
    If \(\Pic^0_C\) is a simple abelian variety, then
    \[\min(\dendegs(C/k)) = \begin{cases}
        g & \textup{if }g \in \calD(C/k), \rk(\Pic^0_C(k)) > 0, \text{ and } \gon(C) \geq g, \\
        \gon(C) & \text{otherwise.}
    \end{cases}\]
\end{cor}

\begin{proof}
    It suffices to consider the case \(\min(\dendegs(C/k)) = \min(\dendegs_{\AV}(C/k))\).  By Remark~\ref{rem:AVsimple}, we have \(\min(\dendegs_{\AV}(C/k)) \geq g\).  If \(\min(\dendegs(C/k)) \geq g+1\), any closed point of degree~\(\min(\dendegs(C/k))\) is \(\PP^1\)-parameterized by Corollary~\ref{cor:HighDegreeP1Param} and so \(\min(\dendegs(C/k)) = \gon(C)\). The remaining case is \(\min(\dendegs(C/k)) = g\),
    which requires \(\gon(C) = \min(\dendegs_{\PP^1}(C/k)) \geq g\) by Remark~\ref{rem:gon_P1param}, and \(g \in \calD(C/k)\) and \(\rk(\Pic^0_C(k)) > 0\) by Lemma~\ref{lem:AVDensityDegreeLarge}.
\end{proof}


\subsection{The density degree set under extensions and the potential density degree set}

The density degree set \(\dendegs(C/k)\) of a curve over a number field can become larger after passing to a finite extension
\(k'/k\) of number fields (e.g. if \(C = V(x^2 + y^2 + z^2) \subset \PP^2\), \(\dendegs(C/\Q) = 2\Z\) and \(\dendegs(C/\Q(i)) = \Z\)).  Propositions~\ref{prop:DenDegLarge} and~\ref{prop:DensdegsClosedUnderScalarMult} imply that there is a maximal growth possible, or, equivalently, that the potential density degree set agrees with the density degree set after passing to a fixed finite extension.

\begin{prop}\label{prop:dendegs_extension}
    Let \(C\) be a nice curve defined over a number field \(k\), and let \(k'/k\) be a finite Galois extension.  Then
    \[\dendegs(C/k) \subset \dendegs(C/k').\]
    In particular, there exists a finite extension \(k_0/k\) such that \(\potdendegs(C/k) = \dendegs(C/k_0)\).
\end{prop}
\begin{proof}
    Let \(d\in \dendegs(C/k)\).  Then by Corollary~\ref{cor:dendeg_parameterized}, there exists a degree~\(d\) parameterized point \(x\in C_k\).  If \(x_{k'}\) is irreducible, then it is a degree~\(d\) point on \(C_{k'}\) and Definitions~\ref{def:P1param} and~\ref{def:AVparam} imply that \(x_{k'}\) is parameterized.

    Thus, we have reduced to the case that \(x_{k'}\) is reducible, and further, we may assume that the above basechange construction results in reducible divisors for all (infinitely many) degree~\(d\) parameterized points \(x\in C_k\).  Then by Proposition~\ref{prop:SplittingUnderExtns}\eqref{part:Galois}, each \(x_{k'}\) contains a degree~\(e\) point for some \(e|d\) (that depends on \(x_{k'}\)).  By the pigeonhole principle, there must be infinitely many \(x_{k'}\) that contain a degree~\(e\) point for the same fixed \(e|d\).  In particular, \(e\in \dendegs(C/k')\), so Proposition~\ref{prop:DensdegsClosedUnderScalarMult} implies that \(d\in \dendegs(C/k')\).

    The containment \(\delta(C/k) \subset  \delta(C/k')\) implies that the potential density degree set can also be defined as the union of the density degree set \(\delta(C/k')\) for all Galois extensions \(k'/k\). Furthermore, there is some Galois extension \(k_1/k\) over which \(C(k_1)\neq\emptyset\), so in particular, \(\ind(C/k_1) = 1\). By Proposition~\ref{prop:DenDegLarge}, \(\dendegs(C/k_1)\supset \Z_{\geq\max(2g(C),1)}\). Thus, for all Galois extensions \(k'/k\) that contain \(k_1\), there are only finitely many possibilities for \(\dendegs(C/k')\).
    Thus the union 
    \[
        \bigcup_{k'/k \textup{ Galois}, k'\supset k_1}\dendegs(C/k')
    \]
    reduces to a finite union. By taking the compositum of the finitely many extensions that give this union, we obtain \(k_0/k\), a finite Galois extension such that \(\potdendegs(C/k) = \dendegs(C/k_0).\)
\end{proof}

    A reader encountering these objects for the first time may reasonably wonder why the containment \(\dendegs(C/k)\subset  \dendegs(C/k')\) in Proposition~\ref{prop:dendegs_extension} is not immediate, even without the Galois hypothesis. Indeed, if \(x\in C\) is a \(\PP^1\)-parameterized point then by Lemma~\ref{lem:P1paramDefineorem} there is a nonconstant morphism $\phi\colon \PP^1 \to \Sym^dC$ whose image contains \(x\in \Sym^d(k)\), and such a morphism retains these properties over base change to any extension \(k'\).  However, the key point is that the base change \(x_{k'}\) may no longer be a \emph{point}, i.e., it may become reducible. 

    For \(\PP^1\)-parameterized points, we may leverage Hilbert's irreducibility theorem to show that even though \(x\) may become reducible over an extension, there will be some other point in \(\phi(\PP^1(k'))\) that corresponds to a degree~\(d\) point.
    \begin{lemma}\label{lem:P1dendeg_extensions}
        Let \(C\) be a nice curve defined over a number field \(k\), let \(x\in C\) be a degree~\(d\) \(\PP^1\)-parameterized point and let \(\phi\colon \PP^1 \to \Sym^d_C\) be the morphism corresponding to a degree~\(d\) map \(\pi \colon C \to \PP^1\) with \(x\) as a fiber. Then for any finite extension \(k'/k\), there exists a degree~\(d\) point \(y\in C_{k'}\) such that \(y\in \im\phi_{k'}\). In particular,
        \[
        \dendegs_{\PP^1}(C/k) \subset \dendegs_{\PP^1}(C/k').
        \]
    \end{lemma}
    \begin{proof}
    The base extension of \(\pi\) to \(k'\) is again a degree~\(d\) morphism \(\pi_{k'}\colon C_{k'} \to \PP^1_{k'}\).  The map \(\PP^1_{k'} \to (\Sym^d_C)_{k'}\) sending a point \(p \in \PP^1_{k'}\) to the fiber \(\pi_{k'}^{-1}(p)\) is the base extension \(\phi_{k'}\) of \(\phi\).  By Proposition~\ref{prop:hilbert-curve}, there exists a point \(p \in \PP^1(k')\) such that \(\pi_{k'}^{-1}(p) = y\) is a degree~\(d\) point on \(C_{k'}\).  By construction \(y \in \im \phi_{k'}\).
    \end{proof}
    
    If \(x\) is an AV-parameterized point then there is a positive rank abelian subvariety \(A\subset  \Pic^0_C\) such that \([x] + A\subset  W_C^d\).  As in the case of \(\PP^1\)-parameterized points, many properties are preserved under basechange. Indeed, the abelian variety \(A\) remains positive rank under any finite extension \(k'/k\) and the containment \([x] + A\subset  W_C^d\) is preserved under basechange. The key issue again is that \(x_{k'}\) may be reducible, i.e., may no longer be a point. Moreover, in contrast to the case of \(\PP^1\)-parameterized points, it is in fact possible that \textit{every} point of \([x] + A(k')\) is represented as a reducible divisor over \(k'\).

\begin{example}
    Let \(k' := \Q(\sqrt{-3})\) and let \(E\) be the elliptic curve defined (in simplified Weierstrass form) by \(y^2 = x^3 - \frac{3}{16}x - \frac1{32} - \frac{\sqrt{-3}}{9}\), which is LMFDB curve \texttt{2.0.3.1-283.1-a1}~\cite{lmfdb}*{\href{https://www.lmfdb.org/EllipticCurve/2.0.3.1/283.1/a/1}{2.0.3.1-283.1-a1}}.  The elliptic curve \(E\) has rank \(1\) over \(k'\), with Mordell-Weil group generated by the point \(P := [\frac14 + \frac16\cdot\sqrt{-3}: \frac14 -\frac14\cdot \sqrt{-3}: 1]\), and is not isogenous to its Galois conjugate curve \(E'\) (i.e., \(E\) is not a \(\Q\)-curve). Let \(P'\in E'\) denote the conjugate of \(P\in E\).

    The restriction of scalars \(A := \Res_{k'/k}E\) is an abelian surface with the following properties (see, e.g., ~\cite{Poonen-RatlPoints}*{Section 4.6} for more details):
    \begin{enumerate}
        \item The base change \(A_{k'}\) is isomorphic to \(E\times E'\).
        \item The surface \(A\) is birational to the (singular) surface \(S\subset \PP^4\) given by the vanishing of
        \[
            \frac12w_0^3 + 3w_0^2w_3 + w_0w_1^2 - 3w_0w_2^2 - 16w_3^3 + 144w_3w_4^2, \quad
            \frac89w_0^3 + \frac32w_0^2w_4 + w_0w_1w_2 - 24w_3^2w_4 + 24w_4^3,
        \]
        where the induced birational map \(S_{k'}\dasharrow E\times E'\subset \PP^2_{k'}\times \PP^2_{k'}\) sends a point \([w_i]_i\) to
        \[
\left(\left[4w_3 + 4\sqrt{-3}w_4:
w_1 + \sqrt{-3}w_2 : 4w_0\right], \left[4w_3 - 4\sqrt{-3}w_4:
w_1 - \sqrt{-3}w_2 : 4w_0\right]\right).\footnote{This statement and all other computational claims in this example are verified in \texttt{Magma} code~\cite{magma} available on \texttt{Github}~\cite{VV-code}.}
        \]
        \item The geometric N\'eron-Severi group of \(A\) is \(\Z\times \Z\), the action of \(G_k\) factors through \(\Gal(k'/k)\simeq \Z/2\), and the involution interchanges the two factors.
    \end{enumerate}

    Define \(C\subset S\) to be the vanishing of 
    \[
    w_4 - w_3\quad \textup{and}\quad 18w_1w_2 + 27w_4w_0 + 16w_0^2;
    \]
    One can check that \(C\subset \PP^4\) is a smooth genus \(4\) curve (given as a complete intersection of a hyperplane, a quadric, and a cubic), and that the integral curve \(Z\subset A\) corresponding to \(P\times E'\cup E\times P'\) pulls back on \(C\) to give a degree~\(4\) \emph{point} \(x\in C\) that is \(\PP^1\)-isolated. Thus we have a map induced by pullback \(\Pic^{(1,1)}_A \to W^4_C= \Pic^4_C\), and the image contains the class \([x]\). In particular, \(x\) is a \(\PP^1\)-isolated AV-parameterized point, so by Proposition~\ref{prop:P1isolAVparam}, a finite index subgroup of \({\im}\left(\Pic^{(1,1)}_A(k)\right)\) corresponds to degree~\(4\) points on \(C\) over \(k\).

    However, over \(k'\), the image of \(\Pic^{(1,1)}_A(k')\to W^4_C(k')\) consists entirely of divisors with reducible representatives, as we now explain. Indeed, since \(A_{k'}\simeq E\times E'\) and \(E\) has rank \(1\), we have the isomorphism:
    \[
    \Z\times \Z \simeq \Pic^{(1,1)}_A(k'), \quad (n,m)\mapsto [nP]\times E' \cup E \times [mP']. 
    \]
    Thus, \(\Pic^{(1,1)}_A(k')\) consists entirely of reducible curves which then pullback to reducible effective divisors on \(C\). If the corresponding divisor class has \(\hh^0 = 1\) (which occurs generically, since \(x\) is \(\PP^1\)-isolated), then it cannot be represented by a degree~\(4\) point by Lemma~\ref{lem:P1isolatedIrreducible}. 
    \hfill\(\righthalfcup\)
\end{example}

\section{Single sources of low degree parameterized points}\label{sec:SingleSource}

    Parameterized points whose degree is significantly smaller than the genus of the curve are rare. Indeed, if they do exist, then they arise as pullbacks under a single morphism.

\begin{thm}\label{thm:single_source}
    Let \(C\) be a nice curve of genus \(g\).
    \begin{enumerate}
        \item There exists a nice curve \(C_0/k\) and a finite map \(\pi \colon C \to C_0\) with degree at least \(2\) such that for every \(\PP^1\)-parameterized point \(x \in C\) of degree less than \(\sqrt{g} + 1\), the image \(\pi(x) = z \in C_0\) is a \(\PP^1\)-parameterized point of degree~\(\deg(x)/\deg \pi\) and \(\pi^*|z| = |x|\). \label{part:P1param}
        \item Moreover, for every \(\AV\)-parameterized, \(\PP^1\)-isolated point \(x\in C\) that 
        has \(\deg(x) < \frac12(\sqrt{g} + 1)\), the image \(\pi(x) = z\) is an \(AV\)-parameterized point of degree~\(\deg(x)/\deg \pi\) and if \(A\in \Pic^0_C\) is such that \([x]+A\subset W^d_C\), then \(A = \pi^*A_0\) where \([z] + A_0 \subset W^d_{C_0}\).\label{part:AVparam}
        \item \label{part:dendegs}In particular,
    \begin{align*}
    \dendegs_{\PP^1}(C/k)\cap (1, \sqrt{g} + 1) & = \deg(\pi)\cdot \left(\dendegs_{\PP^1}(C_0/k)\cap \left[1, \frac{\sqrt{g} + 1}{\deg \pi}\right)\right), \quad \textup{and}\\
    \dendegs_{\AV}(C/k)\cap \left(1, \frac12(\sqrt{g} + 1)\right) & \subset  \deg(\pi)\cdot \left(\dendegs_{\AV}(C_0/k)\cap \left[1, \frac{\sqrt{g} + 1}{2\deg \pi}\right)\right).
    \end{align*}
    \end{enumerate}
\end{thm}
    \begin{remark}
        Theorem~\ref{thm:single_source} guarantees that all but finitely many very low degree parameterized points \(x\) are pullbacks of lower degree points by the single map \(\pi\). This strengthens \cite{Derickx-SingleSource}*{Corollary 1.3}, which showed that all but finitely many low degree points are pulled back under one of finitely many maps.  
        
        Since \(\kk(\pi(x))\) is a subfield of \(\kk(x)\), if \(\pi(x)\) is not a rational point, then the residue field \(\kk(x)\) is an \defi{imprimitive} extension.\footnote{A finite extension \(L/k\) is \defi{primitive} if it has no nontrivial proper subfields and \defi{imprimitive} otherwise. We say a closed point \(x\in C\) is primitive or imprimitive if its residue field extension \(\kk(x)/k\) is.} In other words, if \(\min\dendegs(C/k) < d < \frac12(\sqrt{g} + 1)\), there are \emph{no} degree~\(d\) primitive parameterized points. 
        For related work on primitive points see \cites{KS-primitive, KS-singlesource, Derickx-LargeDegree, Derickx-SingleSource}.

        A related result of Vojta shows that inequalities involving the arithmetic discriminant imply that for any map \(\pi \colon C \to \PP^1\), all but finitely many points \(x\) with \(\deg{x} < (g-1)/\deg{\pi} + 1\) are \defi{contracted} by \(\pi\), i.e., \(\deg(\pi(x)) < \deg(x)\) \cite{Vojta-arithmeticdiscriminant}.
    \end{remark}

\subsection{Key tools in the proof}\label{sec:single_source_key}

There are two key ingredients in the proof of Theorem~\ref{thm:single_source}.  The first is the bound on the  genus of a curve \(C\) that has a (possibly singular) birational model in a particular numerical class on a smooth surface, which can be derived from the adjunction formula \cite{Hartshorne}*{Chapter II, Proposition 8.20}.  We briefly recall two cases of importance for our argument.  
\begin{itemize}
    \item If \(C\) is birational to a degree~\(d\) plane curve, then \(g(C) \leq {d-1 \choose 2}\).
    \item If \(C_1\) and \(C_2\) are smooth curves of genera \(g_1\) and \(g_2\) and \(C \to C_1 \times C_2\) is birational onto its image, then
    \(g(C) \leq (d_1 - 1)(d_2 - 1) + d_1 g_1 + d_2 g_2\) where \(d_i\colonequals\deg( \pi_i \colon C \to C_i )\).
\end{itemize}
The second bound implies that two maps \(C \to C_1\) and \(C \to C_2\) factor through some common map if the genus of \(C\) is sufficiently large.  This is frequently called the \defi{Castelnuovo--Severi inequality}.  The following identifies a class of maps where we can iteratively apply this.

\begin{lemma}\label{lem:genus_bound}
    Let \(C\) be a nice curve of genus \(g\) and suppose that for \(i=1,2\) we have integers \(2 \leq d_i < \sqrt{g} + 1\) and finite maps to nice curves \(\varphi_i \colon C \to Y_i\) of degree~\(d_i\).  If \(g(Y_i) < \left(\frac{\sqrt{g}+1}{d_i}-1\right)^2\) for \(i=1,2\), then the map \((\varphi_1, \varphi_2) \colon C \to Y_1 \times Y_2\) factors through a nontrivial map \(\varphi_3 \colon C \to Y_3\) of degree~\(2 \leq d_3 < \sqrt{g} + 1\) and \(g(Y_3) < \left(\frac{\sqrt{g}+1}{d_3}-1\right)^2\).
\end{lemma}
\begin{proof}[Proof of Lemma~\ref{lem:genus_bound}]
    Let \(Y_3\subset Y_1\times Y_2\) denote the image of \(C\) under \(\varphi_1, \varphi_2\). Let \(d_3\) denote the degree of the morphism \(\varphi_3\colon C \to Y_3\) induced by \((\varphi_1, \varphi_2)\); then the  projections \(Y_3 \to Y_i\) have degrees \(d_i/d_3\), for \(i = 1,2\). If \(d_i=d_3\) for some \(i\), then \(Y_3 \simeq Y_i\) and the conclusion holds; we may therefore assume that \(d_3\) is a proper divisor of \(\gcd(d_1, d_2)\). Hence, the adjunction formula, the assumed genus bounds, and Lemma~\ref{lem:PositivityClaim} below give:
    \begin{align*}
        g(Y_3) & \leq \left(\frac{d_1}{d_3} - 1\right)\left(\frac{d_2}{d_3} - 1\right) + \frac{d_1}{d_3}g(Y_1) + \frac{d_2}{d_3}g(Y_2) \\
        & < \left(\frac{d_1}{d_3} - 1\right)\left(\frac{d_2}{d_3} - 1\right) + \frac{d_1}{d_3}\left(\frac{\sqrt{g}+1}{d_1}-1\right)^2 + \frac{d_2}{d_3}\left(\frac{\sqrt{g}+1}{d_2}-1\right)^2 < \left(\frac{\sqrt{g} + 1}{d_3} -1\right)^2
    \end{align*}
    Additionally, \(\left(\frac{\sqrt{g} + 1}{d_3} -1\right)^2\leq g\), so \(\phi_3\) is not birational, i.e., \(d_3> 1\).
\end{proof}

\begin{lemma}\label{lem:PositivityClaim}
    Let \(g\in \Z_{\geq0}\) and let \(d_1,d_2\in \Z \cap [2, \sqrt{g} + 1)\). For any positive proper divisor \(d_3|\gcd(d_1,d_2)\), 
    \[
    \left(\frac{d_1}{d_3} - 1\right)\left(\frac{d_2}{d_3} - 1\right) + \frac{d_1}{d_3}\left(\frac{\sqrt{g} + 1}{d_1} - 1\right)^2+ \frac{d_2}{d_3}\left(\frac{\sqrt{g} + 1}{d_2} - 1\right)^2 < \left(\frac{\sqrt{g} + 1}{d_3} - 1\right)^2.
    \]
\end{lemma}
\begin{proof}
    The desired inequality is equivalent to proving that
    \[
    f \colonequals \frac{d_1d_2}{d_3^2}\left(\left(\frac{\sqrt{g} + 1}{d_3} - 1\right)^2 - \left(\frac{d_1}{d_3} - 1\right)\left(\frac{d_2}{d_3} - 1\right) - \frac{d_1}{d_3}\left(\frac{\sqrt{g} + 1}{d_1} - 1\right)^2 - \frac{d_2}{d_3}\left(\frac{\sqrt{g} + 1}{d_2} - 1\right)^2\right) >0.
    \]
    Set \(m \colonequals \min\left(\frac{d_2}{d_3}, \frac{d_1}{d_3}\right) - 2\), \(a := \left|\frac{d_2}{d_3}- \frac{d_1}{d_3}\right|\), and \(\varepsilon := \frac{\sqrt{g} + 1}{d_3} - m - 2 - a\). Note that \(m, a\in \Z_{\geq0}\), \(m + 2 + a = \max\left(\frac{d_2}{d_3}, \frac{d_1}{d_3}\right)\),  and so \(\varepsilon>0\). Since \(f\) is a function in \(\frac{\sqrt{g} + 1}{d_3}, \frac{d_1}{d_3}, \frac{d_2}{d_3}\) and furthermore is invariant under interchanging \(\frac{d_1}{d_3}, \frac{d_2}{d_3}\), we may express \(f\) instead as a polynomial in \(m,a, \varepsilon\). Indeed, one can compute\footnote{\texttt{Magma} code~\cite{magma} verifying this claim is available on Github~\cite{VV-code}.} that as a polynomial in \(m, a,\) and \(\varepsilon\), all coefficients of \(f\) are nonnegative integers and the coefficient of \(\varepsilon\) is positive. Precisely,
    \[f = \varepsilon^2(am + a + m^2 + 2m) + 2\varepsilon(m+1)(a+m+2)^2 + (m+1)a(a+m+2)^2 \geq 8\varepsilon > 0.\qedhere\]
%
%
%
\end{proof}

The second key ingredient is the following, which allows us to show that an AV-parameterized point is the pullback of an AV-parameterized point.

\begin{lemma}\label{lem:AV_pullback}
    Let \(\pi\colon C \to C_0\) be a finite map of nice curves over a number field \(k\) and let \(y\) be a degree \(e \geq 1\) point on \(C_0\). Assume that \(\pi^*y\) is a \(\PP^1\)-isolated  AV-parameterized point on \(C\) and that there exists a positive-dimensional abelian subvariety \(A\subset  \Pic^0_{C}\) 
    such that \(([\pi^*y] + A) \cap \pi^*W^e_{C_0}(k)\) is Zariski dense in \([\pi^*y] + A\).
    Then there exists an inclusion \(A \hookrightarrow \Pic^0_{C_0}\) whose composition with \(\pi^*\) is the identity and such that \([y] + A \subset  W^{e}_{C_0}\). 
\end{lemma}
\begin{proof}
    By Faltings's Subvarieties of AV Theorem (Theorem~\ref{thm:Faltings91}), 
    \(W^e_{C_0}(k) = \bigcup_{i=1}^r (D_i + B_i(k)),\)
    where \(B_i \subset  \Pic^0_{C_0}\) is an abelian subvariety and \(D_i + B_i \subset  W^e_{C_0}\).  By hypothesis
    \[\pi^*\left( \bigcup_{i=1}^r (D_i + B_i) \right) \cap ([\pi^* y] + A) \subset  \Pic^{e \cdot \deg \pi}_C\]
    contains a Zariski dense subset of \([\pi^*y] + A\), and therefore contains \([\pi^* y] + A\).  Furthermore, since \(A\) is irreducible, the locus \([\pi^*y] + A\) cannot be covered by finitely many proper closed subvarieties.  Thus there exists a single \(i \in 1, \dots, r\) such that \([\pi^*y] + A \subset  \pi^*(D_i + B_i)\).  
    Since \(\pi^*y\) is \(\PP^1\)-isolated, it has a unique effective representative. Thus \([y] \in D_i + B_i\) and \(\pi^*|_{D_i + B_i}\) is injective.  Therefore \((\pi^*)^{-1}A  \cap B_i \simeq A\) and under this inclusion of \(A\) into \(\Pic^0_{C_0}\) we have \([y] + A \subset  [y] + B_i \subset  W^e_{C_0}\).
\end{proof}

\subsection{Proof of Theorem~\ref{thm:single_source}}

\begin{proof}
    \textbf{\eqref{part:P1param}:} We will construct \(\pi\) and \(C_0\) by an iterative procedure which will eventually terminate. Let \(x\in C\) be a \(\PP^1\)-parameterized point of degree~\(d <  \sqrt{g} + 1\). Let \(X\) be the image of \(C\) under the map \(C\to \PP^{N}\) given by the complete linear system \(|x|\) (so \(N = \hh^0([x]) - 1\)) and let \(\phi\) denote the induced map \(C \to X\). By construction, for any hyperplane section \(z\subset X\), \(\pi^*|z| = |x|\). If \(\hh^0([x]) = 2\), then \(\deg(\phi) = d>1\). If \(\hh^0([x]) > 2\), then \(X\subset \PP^N\) is degree~\(e := d/\deg \phi<(\sqrt{g} + 1)/\deg\phi\) and is birational to its image under a generic projection \(\PP^N \dasharrow \PP^2\). Hence, 
    \begin{equation}\label{eq:genusbound}
        g(X)\leq {e -1 \choose 2} \leq(e-1)^2 < \left(\frac{\sqrt{g} + 1}{\deg\phi}-1\right)^2 \leq g.
    \end{equation}
    Thus \(X\) and \(C\) are not birational, and so \(\phi\) has degree at least \(2\).

    If, for every \(\PP^1\)-parameterized \(y\in C\) with \(\deg(y) < \sqrt{g} + 1\), the image \(z:= \phi(y)\) has degree~\(\deg(y)/\deg(\pi)\) and \(\phi^*|z| = |y|\), then we may take \(C_0 = X\) and \(\pi_0 = \phi\). If not, then there exists a \(y_1\in C\) with \(\deg(y_1)<\sqrt{g} + 1\) such that \(|y_1|\) is not the pullback of a linear system on \(X\). Applying the construction in the previous paragraph to \(y_1\), we obtain a finite map \(\psi\colon C \to Y\) of degree at least \(2\) such that \(|y_1| = \psi^*|\psi(y_1)|\) and such that \(g(Y) < \left( (\sqrt{g} + 1)/\deg\psi - 1\right)^2\) (see~\eqref{eq:genusbound}). By Lemma~\ref{lem:genus_bound}, the product morphism \((\phi, \psi) \colon C \to X\times Y\) factors through a nontrivial map \(\phi_1\colon C \to X_1\) with 
    \[
        2\leq \deg(\phi_1) < \sqrt{g} + 1 \quad \textup{and}\quad g(X_1) < \left(\frac{\sqrt{g} + 1}{\deg(\phi_1)} - 1\right)^2.
    \]
    By construction, any complete linear system on \(C\) which is equal to the pullback of a linear system on \(X\) or \(Y\) (so, in particular, \(|x|\) and \(|y_1|\)) is also equal to the pullback of a linear system from \(X_1\).

    If the map \(\phi_1\) does not have the properties in the theorem, then we may iterate the argument of the previous paragraph with \(\phi_1\) in place of \(\phi\) and a new \(\PP^1\)-parameterized point \(y_2\) in place of \(y_1\). Doing so, we obtain a sequence of maps
    \[
        C\to X_n \to X_{n-1} \to \dots \to X_1 \to X, 
    \]
    each of degree at least \(2\), and where the composition has degree~\(d\). In particular, the process must terminate, and at that point \(C_0 := X_n\) has the desired properties.

    \textbf{\eqref{part:AVparam}:}
    Let \(x\in C\) be an AV-parameterized, \(\PP^1\)-isolated point
    with \(d\colonequals \deg(x) < \frac12(\sqrt{g} + 1)\).  Let \(A\subset \Pic^0_C\) be a positive-rank abelian subvariety such that \([x] + A\subset W^d\) and \(A(k)\) is Zariski dense in \(A\) (cf.~Remark~\ref{rem:Zariski_dense}).  
    By Proposition~\ref{prop:P1isolAVparam}, there exists a finite index subgroup \(H <  A(k)\) such that every point in \([x] + H\) is represented by a degree~\(d\) point. Since \(H <  A(k)\) is finite-index, the points of \(H\) are Zariski dense in \(A\).

    We show in the next paragraph that for any degree~\(d\) point \(z\in C\) that is not a ramification point of \(\pi\) and such that \([z] \in [x] + H\), there is a point \(y_z \in C_0\) of degree~\(d/\deg(\pi)\) such that \(z = \pi^*(y_z)\). Since there are finitely many ramification points and \(\hh^0\leq 1\) holds generically, such points \(z\) that are also \(\PP^1\)-isolated are Zariski dense in \([x] + A\). By Lemma~\ref{lem:AV_pullback}, there is an inclusion \(A \hookrightarrow \Pic^{d/\deg{\pi}}_{C_0}\) such that \(\pi^*|_{A}\) is an isomorphism and \([y_z] + A \subset  W^{d/\deg{\pi}}_{C_0}\).
    Since \([x] \in [z] + A = \pi^*([y_z] + A)\), there exists \([y] \in [y_z] + A(k)\) such that \([x] = [\pi^*(y)]\) and hence \(x = \pi^*y\), since it is \(\PP^1\)-isolated.  The point \(y\) is AV-parameterized by definition.
    
    By Proposition~\ref{prop:DensdegsClosedUnderScalarMult}, there exists a \(\PP^1\)-parameterized point \(w_z\in C\) such that \([w_z]= [2z]\).  By~\eqref{part:P1param}, \(v_z \colonequals \pi(w_z)\in C_0\) is a closed point of degree~\(2d/\deg{\pi}\) and \(\pi^*|v_z| = |w_z| = |2z|\). Hence, there is a degree~\(2d/\deg{\pi}\) morphism \(\varphi \colon C_0 \to \PP^N\) and some hyperplane \(H_z\subset \PP^N\) such that the scheme-theoretic preimage \(\pi^*(\varphi^*(H_z))\) equals \(2z\). Since \(z\) is not a ramification point of \(\pi\), the preimage \(\varphi^*(H_z) \subset C_0\) must be non-reduced with multiplicity \(2\) everywhere and thus \(\pi^*(\varphi^*(H_z)^{\textup{red}}) = z\).  Hence \(y_z \colonequals \varphi^*(H_z)^{\textup{red}} \in C_0\) is the desired degree~\(d/\deg{\pi}\) point.
\end{proof}

\section{Further directions}\label{sec:Future}

In the previous sections, we have endeavored to lay out the foundations of parameterized points, isolated points, and (potential) density degree sets and illustrate their power. The reader interested in exploring more is in luck; there are many possible future directions to investigate. It is impossible to make a comprehensive list of all possible research directions, so in this section, we content ourselves with highlighting merely a selection of them. Similarly, some of these questions intersect with existing active areas of research that would require papers of their own in which to give comprehensive summaries. So the reader should view each section as a snapshot in time and an invitation to explore the literature in that area further, rather than a comprehensive history and summary of the state-of-the-art.

\subsection{Further structure on density degree sets and potential density degree sets}\label{sec:semigroup}

    In Example~\ref{ex:DenDegsNotSemigroup}, we constructed a curve \(C\) where \(\delta_{\PP^1}(C/\Q) + \delta_{AV}(C/\Q)\not\subset \delta(C/\Q)\). However, a crucial component of that example was that the rational points of \(\Pic^0_C(\Q)\) were contained in a lower dimensional subvariety. This containment no longer holds over all finite extensions. Thus the following question remains open.
    \begin{question}
        Let \(C/k\) be a nice curve over a number field. Is \(\potdendegs(C/k)\) a semigroup? 
    \end{question}
    \noindent Example~\ref{ex:DenDegsNotSemigroup} also leaves open the (albeit unlikely) possibility that
    \(\delta_{AV}(C/k)\) is a semigroup.

\subsection{Classification of curves with small minimum potential density degree}
 
Faltings's Curve Theorem gives a geometric classification of curves \(C\) over a number field \(k\) with \(\min(\potdendegs(C/k)) = 1\): they are precisely the curves of genus \(0\) or \(1\).  This raises the question of classifying curves for which \(\min(\potdendegs)\) takes other (small) values. The paradigm of \(\PP^1\)- and \(\AV\)-parameterized points implies the the following geometric description.

\begin{prop}
    We have
    \(\min(\potdendegs(C/k)) = \min\big(\{\gon(C_{\kbar})\}\cup\left\{ 
    d : \Ueno\left((W^d_C)_{\kbar}\right) \neq \emptyset \right\}
    \big).\)
\end{prop}
\begin{proof}
    Let \(d\) be the quantity on the right. 
    Theorem~\ref{thm:InfinitudeDegreed} over finite extensions \(k'/k\) implies that 
    \(\min(\potdendegs(C/k))\geq d\).
    For the reverse inequality, 
    if \(d = \gon(C_{\kbar})\), then there exists a finite extension \(k'/k\) for which \(d \in \dendegs_{\PP^1}(C/{k'}) \subset  \potdendegs(C/k)\).
    We may therefore assume that \(d < \gon(C_{\kbar})\) is the smallest positive integer for which \(\Ueno\left(({W_C}^d)_{\kbar}\right) \neq \emptyset\).  
    Then there exists a finite extension \(k'/k\) and an abelian translate \(A \subset  \Ueno\left(({W_C}^d)_{k'}\right)\) with \(\rk(A(k')) > 0\).  Thus \(\Sym^d_C(k')\) is infinite but, since \(\min(\dendegs(C/k')) \geq \min(\potdendegs(C/k)) \geq d\), \(\Sym^{d'}_C(k')\) is finite for all \(d' < d\).  Thus \(d \in \dendegs(C/{k'}) \subset  \potdendegs(C/k)\) by Lemma~\ref{lem:image_summation_reducible}.
\end{proof}

The task is to make this explicit.
The results of Harris--Silverman \cite{HS-Degree2} and Abramovich--Harris \cite{AH-Degree3And4} quoted in Section~\ref{sec:ConstructionCoversOfEllipticCurves} give an explicit classification when \(\min(\potdendegs(C/k)) = 2\) or \(3\): over~\(\kbar\) such curves \(C\) are always degree~\(2\) or \(3\) covers of curves of genus \(0\) or \(1\). 
On the other hand, Debarre--Fahlaoui give examples that show that \(\min(\potdendegs(C/k))\) can equal \(4\) even when \(C_{\kbar}\) is not a degree~\(4\) cover of a curve of genus at most \(1\)~\cite{DF-CounterEx}.  In \cite{Kadets--Vogt}, Kadets and the second author extend Harris--Silverman's and Abramovich--Harris's classification results to \(d \in\{4,5\}\): the only curves with \(\min(\potdendegs(C/k)) = d\) without a degree~\(d\) map to a curve of genus at most \(1\) are the examples that Debarre--Fahlaoui constructed!  It is an intriguing problem to extend this to \(d > 5\), where, presumably, examples beyond those constructed by Debarre--Fahlaoui can be found.

\subsection{The number of isolated points}

Given that the set of isolated points is finite, it is natural to ask how large this set can be.  By interpolation, we can always obtain the existence of a curve of genus at least~\(2\) with arbitrarily many rational points, which, by Faltings's Curve Theorem, are necessarily isolated.  However, interpolation yields a curve whose genus grows with the number of rational points.  Thus, a better question is how large the set of isolated points can be on curves \emph{of fixed genus.}

\subsubsection{Towards a uniform (upper) bound}
	Recent work of Dimitrov--Gao--Habegger~\cite{DGH-UniformMordell} combined with work of K\"uhne~\cite{Kuhne} gives a bound on the number of rational points of a curve of genus \(g\) that depends only on \(g\) and the rank of the Jacobian (see~\cite{Gao} for a survey of how the works come together).  Further work of Gao--Ge--K\"uhne gives a uniform version of the Mordell--Lang conjecture, which yields the following bound on non-Ueno isolated points.
\begin{thm}[Corollary of \cite{GGK-UniformMordellLang}*{Theorem 1.1}]\label{thm:Uniform}
    Let \(C\) be a smooth genus \(g\) curve over a number field \(k\).  Then the number of non-Ueno isolated points is bounded by a constant that depends only on \(g\) and the rank of \(\Pic^0_C(k)\).
\end{thm}
\begin{proof}
    By Corollary~\ref{cor:HighDegreeP1Param}, any isolated point has degree at most \(g\). Further, every point of degree~\(g\) is Ueno (Definition~\ref{def:Ueno}). Thus, it suffices to show that for each \(d < g\), the number of non-Ueno isolated points of degree~\(d\) is bounded by a constant that depends only on \(g, d\) and the rank of \(\Pic^0_C(k)\).

    Let \(\theta\) denote an (ample) theta divisor on \(\Pic^0_C\). By  Poincar\'e's Formula~\cite{ACGH}*{Chapter I, Section 5}, 
    \[
    \deg_{\theta} W_C^d = \deg_{\theta}\left(\frac{\theta^{g-d}}{(g-d)!}\right) = \frac{\theta^{g}}{(g-d)!} = \frac{g!}{(g-d)!}.
    \]
    Thus the degree \(\deg_\theta W^d_C\) is a function of \(g\) and \(d\) alone.
    
    The rational points on \(\Sym^d_C\) corresponding to isolated degree~\(d\) points on \(C\) inject into the rational points of \(W_C^d\), so it suffices to show that the images of the non-Ueno isolated points in \(W_C^d(k)\) are bounded by a constant that depends only on \(g, d\) and the rank of \(\Pic^0_C(k)\).  By Faltings's Subvarieties of AV Theorem and~\cite{GGK-UniformMordellLang}*{Theorem 1.1}, there exists 
    \begin{enumerate}
        \item a constant \(c=c(g,d)\),\footnote{In general, in \cite{GGK-UniformMordellLang}*{Theorem 1.1}, this constant depends on the degree of \(W_C^d\), but, as we have observed, \(\deg_{{\theta}} W_C^d\) is a function of \(g\) and \(d\).}
        \item a positive integer \(N \leq c^{1 + \rank \Pic^0_C(k)}\), and 
        \item for each \(1\leq i\leq N\), a point \(x_i\in \Pic^d_C(k)\) and abelian subvarieties \(B_i\subset  \Pic^0_C\) with \(x_i + B_i \subset  W_C^d\),
    \end{enumerate}
    such that
    \(W_C^d(k) = \bigcup_{i = 1}^N (x_i + B_i(k))\).
    Since, by definition, the image of the non-Ueno points must be contained in the union of \(x_i + B_i(k)\) where \(B_i\) is \(0\)-dimensional (and hence \(B_i\) equals the identity), the number of images of the non-Ueno points is bounded by \(N\leq c^{1 + \rank \Pic^0_C(k)}\). 
\end{proof}

Any Ueno isolated point \(x\) of degree~\(d\) lies on a translate of a positive dimensional rank \(0\) abelian variety \(A\subset  \Pic^0_C\) such that \(x + A\subset  W_C^d\).  The theorem of Gao--Ge--K\"uhne bounds the number of such abelian varieties by a constant that depends on \(g,d,\) and \(\rank \Jac_C(k)\).  Since any isolated point has \(d\leq g\), this implies that the number of Ueno isolated points is bounded depending only on \(g\) and \(\rank \Jac_C(k)\) if the size of \(k\)-rational torsion on abelian varieties of dimension at most \(g\) is bounded by a constant that depends only on \(g\) and \(\rank \Jac_C(k)\).
The existence of such a uniform bound remains open.

\subsubsection{Curves with many isolated points}\
    On a curve of genus at least \(2\), all rational points are isolated points. Thus, curves with record breaking numbers of rational points also give curves with large numbers of isolated points. The current record is a genus \(2\) curve \(C/\Q\) with 642 \(\Q\)-points, found by Michael Stoll~\cite{MS-canonicalhts}. Since \(\rk\Jac(C)=22\), the set of isolated points is exactly the set \(C(\Q)\) by Lemma~\ref{lem:AVDensityDegreeLarge}.
    However, for genus at least \(3\), it seems possible that the record for number of isolated points is larger than the number of rational points.

    To compute isolated points, one first needs to compute the set of rational points on \(W_C^d\). If the Jacobian of the curve has rank \(0\), then computing the rational points on \(W_C^d\) can be deduced from a description of rational points on \(\Pic^0_C\) using Riemann--Roch spaces. If the Jacobian of the curve has positive rank, then this is a much more difficult problem, which is an active area of current research. Many current techniques rely on \defi{symmetric Chabauty}; see~\cite{Siksek} for more details.
    
\subsubsection{Average number of isolated points} The previous discussions about the number of isolated points focus on extreme behavior. It is also important to understand the typical or average behavior, which brings us to the realm of arithmetic statistics. Recent work of Laga and Swaminathan shows that for a fixed \(g\geq 4\), a positive proportion of hyperelliptic curves with a rational Weierstrass point have no \(\PP^1\)-isolated non-Weierstrass points~\cite{LS-UnexpectedQuadratic}.

    \subsection{Arithmetic properties of parameterized points}
While the density degree set \(\dendegs(C/k)\) gives some arithmetic information about the curve, one can ask for more refined information. For example, given \(d\in \dendegs(C/k)\), one can ask which Galois groups of degree \(d\) points appear infinitely often, or which local splitting behaviors occur infinitely often. These questions are just beginning to be explored~\cites{KS-primitive, KS-singlesource, Derickx-LargeDegree, Derickx-SingleSource, Rawson-cyclic, BCLV}. The intersection of these questions is intimately connected to active and long-running research areas motivated by the inverse Galois problem and refinements, such as the Grunwald problem, the Beckmann-Black problem, and the parameterization problem; see, e.g.,~\citelist{\cite{Wittenberg-PCMI}; \cite{Debes-BeckmannBlack}; \cite{KN-HGspecialization}, \cite{JLY}} respectively and the references therein for further details.

\appendix

\section{Hilbert's Irreducibility Theorem and extensions}\label{app:HilbertIrred}
  
In this appendix, we outline the proof of the following proposition.
    \begin{prop}\label{prop:hilbert}
        Let \(X\) be an irreducible variety of dimension \(n \geq 1\) defined over a number field \(k\), and suppose that there exists a dominant map \(\pi \colon X \to \PP_k^n\), which is generically of degree~\(d\).  Then there exists a Zariski dense subset of points \(t \in \PP^n(k)\) such that \(\pi^{-1}(t)\) is a degree~\(d\) point on \(X\).  
        In particular, \(d \in \dendegs(X/k)\).
    \end{prop}
    
    The standard approach to proving Proposition~\ref{prop:hilbert} is based on Galois theory ideas discussed in Section~\ref{subsec:GaloisTheory}: if you can show that for a Zariski dense subset of \(Y(k)\), the fiber has Galois group which is a transitive subgroup of \(S_d\), then it is a degree~\(d\) point. 

    We give a brief overview of the ingredients going into the proof of Proposition~\ref{prop:hilbert}, with references to where the interested reader can find more details. Then we show how to use the same approach to prove Proposition~\ref{prop:Hilbert_E}.

    \begin{defn}[cf.~\cite{SerreMW}*{Section 9.1}]
        
    A subset \(\Omega \subset \PP^n(F)\) is called \defi{thin} if it is contained in a finite
    union of subsets of the following two types:
    \begin{enumerate}
        \item (\defi{type 1 thin set}) \(V(F)\) for \(V \subset \PP_F^n\) a proper closed subvariety,
        \item (\defi{type 2 thin set}) \(\pi(X(F))\) for \(X\) an irreducible variety and \(\pi \colon X \to \PP_F^n\) a dominant map of degree at least~\(2\).
    \end{enumerate}  
    \end{defn}
	\noindent The following result on specialization of Galois groups shows the relevance of thin sets to Proposition~\ref{prop:hilbert}.
    
    \begin{lemma}[Specialization of Galois groups (cf.~\cite{SerreMW}*{Section 9.2})]\label{lem:Gal_specialization}
        Let \(Y/F\) be an irreducible variety and let \(\pi \colon Y \to \PP^n\) be a dominant map with \(\kk(Y)/\kk(\PP^n)\) a finite Galois extension.  There is a thin subset \(\Omega \subset \PP^n(F)\), such that for all \(t \in \PP^n(F)\smallsetminus \Omega\), and any closed point \(x \in \pi^{-1}(t)\), the extension \(\kk(x)/F\) is Galois with Galois group \(\Gal(\kk(Y)/\kk(\PP^n))\).
    \end{lemma}
    The idea of the proof of this lemma is that the locus in \(\PP^n(F)\) over which the Galois group is a proper subgroup \(H \subsetneq \Gal(\kk(Y)/\kk(\PP^n))\) is contained in the image of the \(F\)-points of \(Y/H\) under the quotient map \(Y/H \to \PP^n\).
    
    \begin{defn}
    A field \(F\) is called \defi{Hilbertian} if for \(n\geq 1\), the set \(\PP^n(F)\) itself is \emph{not} a thin set.  (This is equivalent to requiring only that \(\PP^1(F)\) is not thin.) 
    \end{defn}
            Fields \(F\) for which \(F^\times/F^{\times2}\) is finite, like algebraically closed fields, finite fields, or local fields, are not Hilbertian, since \(\PP^1(F)\) is contained in a finite union of type 2 thin sets coming from \(\PP^1 \xrightarrow{t \mapsto a t^2} \PP^1\) for \(a \in F^\times/F^{\times2}\).

    \begin{lemma}\label{lem:comp_thin_dense}
        Suppose that \(F\) is Hilbertian and that \(\Omega \subset \PP^n(F)\) is a thin set.  Then \(\PP^n(F) \smallsetminus \Omega\) is Zariski dense in \(\PP^n\).
    \end{lemma}
    \begin{proof}
        The set \(\PP^n(F) \smallsetminus\Omega\) cannot be contained in a proper closed subvariety \(V \subsetneq \PP^n\), since \(V(F) \cup \Omega\) is a thin set.
    \end{proof}
    
    \begin{thm}
    \label{thm:numfield_hilbertian}
        A number field \(k\) is Hilbertian.
    \end{thm} 
    
    \noindent See \cite{SerreMW}*{Section 9.4}  for the reduction to \(k=\Q\) and \cite{SerreMW}*{Section 9.6} for a proof of this result when \(k=\Q\).  For a proof that works in greater generality see \cite{FJ-FieldArithmeticThird}*{Theorem 13.4.2}.

    \begin{proof}[Proof of Proposition~\ref{prop:hilbert}]
        Let \(G\) be the Galois group of the Galois closure of \(\kk(X)/\kk(\PP^n)\).  Since \(X\) is irreducible, \(G\) is a transitive subgroup of \(S_d\).
        Spread out the Galois closure of \(\kk(X)/\kk(\PP^n)\) to a morphism \(Y \to \PP^n_k\) that factors \(Y \to X \xrightarrow{\pi} \PP^n\). By Lemma~\ref{lem:Gal_specialization}, there exists a thin set \(\Omega \subset \PP^n(k)\) such that for all \(t \in \PP^n(k) \smallsetminus \Omega\), the fiber \(Y_t\) is a single closed point with Galois group over~\(k\) equal to \(G\).  In particular, \(X_t\) must also be a single closed point of degree~\(d\) for all \(t \in \PP^n(k) \smallsetminus \Omega\).  Since \(k\) is Hilbertian by Theorem~\ref{thm:numfield_hilbertian}, it follows from Lemma~\ref{lem:comp_thin_dense} that \(\PP^n(k) \smallsetminus \Omega\) is Zariski dense in \(\PP^n\).
    \end{proof}

\subsection{Proof of Proposition~\ref{prop:Hilbert_E}}

To prove Proposition~\ref{prop:Hilbert_E}, let us try to imitate the proof of Proposition~\ref{prop:hilbert} in this context: let
\(\widetilde{C} \to C \xrightarrow{\pi} E\) be a Galois cover with Galois group \(G\).  
If we could show that for any nontrivial subgroup \(H \subsetneq G\), the set \(E(k) \smallsetminus \im((\widetilde{C}/H)(k))\) is Zariski dense, then we could conclude as in the proof of Proposition~\ref{prop:hilbert}.  This is immediate by Faltings's Curve Theorem if \(\widetilde{C}/H \to E\) is a ramified cover, since then \(\widetilde{C}/H\) is a nice curve of genus at least \(2\).  However, even though our original map \(\pi \colon C \to E\) does not factor through a nontrivial connected \'etale cover, the same does not need to be true of a Galois closure \(\widetilde{\pi} \colon \widetilde{C} \to E\).  
We will reduce the proof of Proposition~\ref{prop:Hilbert_E} to the following special case where the above strategy immediately succeeds.

\begin{lemma}\label{lem:Hilbert_E_special}
    Let \(\pi \colon C \to E\) be a degree~\(d\) cover of a positive rank elliptic curve by a nice curve.  Suppose that \(\widetilde{C} \to C \to E\) is a Galois cover such that for any nontrivial subcover \(\widetilde{C} \to C'\to E\) either
    \[
     C'\to E \textup{ is ramified,}\quad \textup{or} \quad C'(k) = \emptyset.
    \]
There exists a Zariski dense locus of points \(t \in E(k)\) for which \(\pi^{-1}(t)\) is a degree~\(d\) point.
\end{lemma}
\begin{proof}
    Let \(G = \Gal(\kk(\widetilde{C})/\kk(E))\) and let \(H \subsetneq G\).  By assumption, either \(\widetilde{C}/H \to E\) is ramified, or \((\widetilde{C}/H)(k) = \emptyset\).  In either case, we have \(\#(\widetilde{C}/H)(k) <  \infty\).  Taking the finite union over possible \(H\), we obtain \(T \colonequals E(k) \smallsetminus \bigcup_{H \subsetneq G} \im((\widetilde{C}/H)(k))\) is infinite and hence Zariski dense.  For all \(t \in T\), the fiber \(\widetilde{C}_t\) is a single closed point of degree~\(d\) (and \(\Gal(\kk(\widetilde{C}_t)/k) = G\)), and hence the same is true of \(C_t\).
\end{proof}

To deduce Proposition~\ref{prop:Hilbert_E} from Lemma~\ref{lem:Hilbert_E_special}, we need a criterion to be able to witness when a morphism factors through a nontrivial connected \'etale cover.

\begin{lemma}[cf.~\cite{cdjlz}*{Lemma 4.4}] \label{lem:not_factoring_equivalent}
    Let \(C/k\) be a nice curve and let \(\pi \colon C \to E\) be a finite cover of an elliptic curve \(E/k\).  The following are equivalent:
    \begin{enumerate}
    \item\label{part:dnf} The cover \(\pi\) does not factor through a nontrivial connected \'etale cover.
    \item For all connected \'etale covers \(\psi \colon C' \to E\), the fiber product \(\psi^*C \colonequals C \times_E C'\) is connected.
    \end{enumerate}
\end{lemma}

\begin{remark}\label{rem:base_extend}
Since the union of Galois conjugates of an \'etale cover defined over a nontrivial extension \(k'/k\) is a connected \'etale cover defined over \(k\), part \eqref{part:dnf} holds over \(\kbar\) if and only if it holds over \(k\). 
\end{remark}

\begin{proof}
Given a connected \'etale cover \(\psi\colon C' \to E\), the pullback \(\psi^*C' = C' \times_E C'\) is not connected since the diagonal \(C'\) is a nontrivial connected component.
If \(\pi\) factors through a connected \'etale cover \(\psi\) then it follows that \(\psi^*C\) also cannot be connected, since it dominates a disconnected curve.

Conversely, assume that there exists a nontrivial connected \'etale cover \(\psi \colon C' \to E\) of degree~\(N > 1\) for which \(\psi^*C\) is disconnected.  Let \(X \subsetneq \psi^*C\) be a nontrivial connected component, and assume that \(X\) maps with degree~\(m < N\) onto \(C\).  Call this map \(\varphi \colon X \to C\).
\begin{equation}\label{eq:diagX}
    \begin{tikzcd}
        {\psi}^* C \supsetneq X \arrow[r, "\text{degree~\(m\)}" {swap}, "\varphi"] \arrow[d, "f"] & C \arrow[d, "\pi"]\\
        C' \arrow[r, "{\psi}", "\text{degree~\(N\)}" {swap}] & E
    \end{tikzcd}
\end{equation}
 Consider the \(m\)th relative symmetric power \(\Sym^m_{C'/E}\) defined as the pullback of \(\Sym^m_{C'} \to \Sym^m_E\) along the diagonal \(E \hookrightarrow\Sym^m_E\).  This is a finite (not necessarily connected) \'etale cover of \(E\). 
 
 Further, we claim that \(\Sym^m_{C'/E}\to E\) has no section. Assume otherwise and let \(\sigma \colon E \to \Sym^m_{C'/E}\) be a section.  Then, we may consider the morphism \(\sigma^*\mathcal{U}^m_{C'}\to C'\). Since \(m<N\), this morphism is not surjective. On the other hand, the composition \(\sigma^*\mathcal{U}^m_{C'}\to C'\to E\) must be surjective for \(\sigma\) to be a section.  Since \(C'\) is connected, this gives a contradiction.

Let \(x\) be a point of \(\varphi^{-1}(p)\) for some \(p \in C\).  Since the diagram \eqref{eq:diagX} is commutative, we have \(\pi(\varphi(x)) = \psi(f(x))\). In particular, the points of \(f(\varphi^{-1}(p))\) lie in the same fiber of \(\psi\).
Since \(\varphi\) has degree~\(m\), we obtain a map \(C \to \Sym^m_{C'/E}\) factoring \(\pi\).
  Since \(C\) is connected, the image is a single connected component \(C'' \subset  \Sym^m_{C'/E}\).  The map \(C'' \to E\) is a nontrivial connected \'etale cover since \(\Sym^m_{C'/E} \to E\) is \'etale without a section.  The map \(\pi\) therefore factors \(C \to C'' \to E\), as desired. 
\end{proof}

\begin{proof}[Proof of Proposition~\ref{prop:Hilbert_E}]
Let \(\varphi \colon E' \to E\) be the maximal connected \'etale cover with the property that \(E'(k) \neq \emptyset\) through which \(\widetilde{C} \to E\) factors; that is, we have a factorization \(\widetilde{C} \to E' \xrightarrow{\varphi} E\), and \(\widetilde{C} \to E'\) does not factor through any nontrivial connected \'etale covers that have \(k\)-points.  Pullback the cover \(\pi \colon C \to E\) by \(\varphi\):
\begin{center}
\begin{tikzcd} 
       & \varphi^* C \arrow[d] \arrow[r] & E' \arrow[d, "\varphi"]\\
     \widetilde{C} \arrow[r] \arrow[ru, bend left = 30] &C \arrow[r, "\pi"] & E
\end{tikzcd}
\end{center}
Since \(C \to E\) does not factor through a nontrivial connected \'etale subcover, the pullback \(\varphi^*C\) is connected by Lemma~\ref{lem:not_factoring_equivalent}.  Hence, the map \(\widetilde{C} \to \varphi^*C\) is dominant.
Since \(\widetilde{C} \to E\) is Galois, so is \(\widetilde{C} \to E'\).  Furthermore, up to possibly changing the origin on \(E\) by translation, we may assume that \(E' \to E\) is an isogeny.  Hence, \(E'\) is again an elliptic curve with positive rank.  It now follows by Lemma~\ref{lem:Hilbert_E_special}
that there exists a Zariski dense set \(T' \subset  E'(k)\) for which \((\varphi^*C)_t\) is a degree~\(d\) point for all \(t \in T'\).  We have \(C_{\varphi(t)} = \varphi^*C_{t}\), so the Proposition holds for the Zariski dense set \(\varphi(T) \subset  E(k)\).
\end{proof}

\begin{remark}
   The essential property about (positive rank) elliptic curves that we used was that for any \emph{ramified} cover \(f \colon Y \to E\), the locus \(E(k) \smallsetminus f(Y(k))\) was Zariski dense.  This was immediate from Faltings's Curve Theorem.  

   Generalizing to higher-dimensional bases, this becomes a highly nontrivial statement.  A variety \(X/k\) is said to satisfy the \defi{weak Hilbert property over \(k\)} if for any finite collection of finite ramified morphisms \(f_i \colon Y_i \to X\) for \(Y_i\) integral and normal, the locus \(X(k) \smallsetminus \bigcup_i f_i(Y_i(k))\) is Zariski dense.  In \cite{cdjlz} it is proved that abelian varieties of arbitrary dimension over a number field \(k\) with Zariski dense \(k\)-points satisfy the weak Hilbert property over \(k\), and the higher-dimensional generalization of Proposition~\ref{prop:Hilbert_E} holds \cite{cdjlz}*{Theorem 1.4}.
\end{remark}

\section{Asymptotics of the density degree set on varieties of arbitrary dimension}\label{sec:Asymptotics}\label{app:GCDdendegs}
  

In this appendix we show that asymptotic results about the density degree set of a curve over a number field as in Proposition~\ref{prop:DenDegLarge} imply similar statements for all positive-dimensional nice varieties defined over a number field.  The claim about degree sets appears (in a slightly more general context) in \cite{GabberLiuLorenzini}*{Proposition 7.5}; below we show that the strategy of proof in op.cit. implies the same result about the \emph{density} degree set as well.

\begin{prop}\label{prop:large_degree_density}
    Let \(X\) be a nice variety of dimension at least \(1\) defined over a number field \(k\).  The degree set \(\calD(X/k)\) and the density degree set \(\dendegs(X/k)\) contain all sufficiently large multiples of the index \(\ind(X/k)\).  In particular, \(\ind(X/k) = \gcd(\dendegs(X/k))\).
\end{prop}
\begin{proof}
    Let \(n\) be the dimension of \(X\).
     By definition of the index we can find finitely many distinct closed points \(P_1, \dots, P_r \in X\) such that \(\gcd(\deg(P_1), \dots, \deg(P_r)) = \ind(X/k)\).
     Let \(Z = P_1 + \cdots + P_r\) and write \(X' \colonequals \operatorname{Bl}_ZX\) for the blowup of \(X\) along \(Z\).  Let \(E_i\) denote the exceptional divisor over \(P_i\) and let \(E \colonequals E_1 + \cdots + E_r\) denote the total exceptional divisor.  (Since \(P_i\) is a smooth point of an \(n\)-dimensional variety, we have \(E_i \simeq \PP^{n-1}_{\kk(P_i)}\).)  Let \(L\) be an ample line bundle on \(X\).  By {\cite{Hartshorne}*{Proposition 7.10(b)}}, for \(d \gg 0\), the complete linear system of \(\beta^*L^{\otimes d}(-E)\) defines a nondegenerate embedding \(X' \hookrightarrow \PP_k^N\).  Since \(\beta^*L^{\otimes d}(-E)|_{E_i} \simeq \OO_{X'}(-E_i)|_{E_i} \simeq \OO_{\PP^{n-1}_{\kk(P_i)}}(1)\), the exceptional divisors are linearly embedded.

    By Bertini's irreducibility and smoothness theorems, there is a nonempty Zariski open \(U \subset (\check{\PP}^N)^{n-1}\) such that for every tuple \((H_1, \dots, H_{n-1}) \in U(k)\) the intersection \(X' \cap H_1 \cap \cdots \cap H_{n-1}\) is a smooth and geometrically irreducible curve.  The genus \(g\) of the curve \(X' \cap H_1 \cap \cdots \cap H_{n-1}\) is independent of the choice of \((H_1, \dots, H_{n-1}) \in U(k)\). Furthermore, \( E_i \cap H_1 \cap \cdots \cap H_{n-1}\) is a \(0\)-dimensional linear section of \(\PP^{n-1}_{\kk(P_i)}\) for all \(i\) and therefore is a single point with residue field \(\kk(P_i)\).  In particular \(\ind(X' \cap H_1 \cap \cdots \cap H_{n-1}/k)\) divides \(\gcd(\deg(P_1), \dots, \deg(P_r)) = \ind(X/k)\).  

    Thus Proposition~\ref{prop:DenDegLarge} guarantees that there is uniform \(m_0\) such that for every tuple \((H_1, \dots, H_{n-1}) \in U(k)\) and every \(m\geq m_0\), we have \(m \ind(X/k) \in \dendegs(X' \cap H_1 \cap \cdots \cap H_{n-1}/k)\).  We claim that by varying the tuples \((H_1, \dots, H_{n-1})\), we obtain sufficiently many curves to 
    cover a dense open subset of \(X'\).  Since \(X\) and \(X'\) are birational, the same follows for \(X\). Thus, for all \(m\geq m_0\), \(m\ind(X/k) \in \delta(X/k)\).
     
    Now consider the following diagram.
\begin{center}
    \begin{tikzcd}[row sep  = small]
        &
        I_{n-1} \colonequals \{(p, (H_1, \dots, H_{n-1})) \in \PP^N \times (\check{\PP}^N)^r : p \in H_1 \cap \cdots \cap H_r\} \arrow[dl, "\pi_1"] \arrow[dr, "\pi_2", swap] & \\
        \PP^N && (\check{\PP}^N)^{n-1}
    \end{tikzcd}
\end{center}
Proving the claim is equivalent to showing that \(\pi_1(\pi_2^{-1}(U))\) contains a dense open subset of \(X\). Notice that \(I_{n-1}\) is the \((n-1)\)-fold fiber product over \(\PP^N\) of the universal hyperplane 
\(I \colonequals \{(p, H) \in \PP^N \times \check{\PP}^N : p \in H\}\).
Since \(I \to \PP^N\) is a projective bundle, the map \(I_{n-1} \to \PP^N\) is flat (and every fiber is isomorphic to a product of projective spaces). Since \(X\) is projective of dimension \(n\), for every  \((H_1, \dots, H_{n-1}) \in (\check{\PP}^N)^{n-1}\) there exists a point \(p \in X \cap H_1 \cap \cdots \cap H_{n-1}\).  In other words, the map \(\pi_2|_{\pi_1^{-1}(X)}\) is surjective. The preimage \(\pi_2|_{\pi_1^{-1}(X)}^{-1}(U)\) of a nonempty open \(U \subset (\check{\PP}^N)^{n-1}\) is therefore a nonempty open in \(\pi_1^{-1}(X)\).  Since \(\pi_1|_{X}\) is flat (and hence open), the image \(\pi_1(\pi_2|_{\pi_1^{-1}(X)}^{-1}(U))\) is a nonempty open in \(X\).  Since \(X\) is (geometrically) irreducible, this nonempty open is dense.
\end{proof}


\end{document}